\documentclass[10pt]{amsart}
\usepackage{amssymb,amsfonts}
\usepackage[normalem]{ulem}
\usepackage[all,arc]{xy}
\usepackage{enumitem}
\usepackage{mathrsfs}
\usepackage{ stmaryrd }
\usepackage{url}
\usepackage{tikz-cd}

\newcommand{\rob}{\color{blue}}
\newcommand\robout{\bgroup\markoverwith {\textcolor{blue}{\rule[0.5ex]{2pt}{0.4pt}}}\ULon} 

\newcommand\prestonout{\bgroup\markoverwith {\textcolor{magenta}{\rule[0.5ex]{2pt}{0.4pt}}}\ULon} 

\usepackage
[pdfauthor={Robert Pollack and Preston Wake},
 pdftitle={two variable mu},
 bookmarks=false]
{hyperref}

\dedicatory{In memory of Joël Bellaïche}

\newcounter{qcounter}

\define\bP{\mathbb{P}}

\define\isoto{\xrightarrow{\sim}}
\define\onto{\twoheadrightarrow}
\define\coker{\mathrm{coker}}
\define\Spec{\mathrm{Spec}}

\define\sF{\mathcal{F}}
\define\E{\mathcal{E}}

\newcommand{\dia}[1]{\langle #1 \rangle}

\newcommand{\ttmat}[4]{\left( \begin{array}{cc}
#1 & #2 \\
#3 & #4
\end{array}
\right)}

\newcommand{\Z}{\mathbb{Z}}
\newcommand{\Q}{\mathbb{Q}}
\newcommand{\C}{\mathbb{C}}

\newcommand{\F}{\mathbb{F}}

\newcommand{\fH}{\mathfrak{H}}
\newcommand{\h}{\mathfrak{h}}

\newcommand{\sO}{\mathcal{O}}

\newcommand{\I}{\mathcal{I}}
\newcommand{\p}{\mathfrak{p}}
\newcommand{\m}{\mathfrak{m}}

\newcommand{\cL}{\mathcal{L}}

\newcommand{\Hom}{\mathrm{Hom}}

\newcommand{\End}{\mathrm{End}}

\newcommand{\zinf}{\{0,\infty\}}

\DeclareMathOperator{\rank}{rank}

\newcommand{\Qinf}{\Q_{\infty}}

\newtheorem{thm}{Theorem}[section]

\newtheorem{prop}[thm]{Proposition}
\newtheorem{lem}[thm]{Lemma}

\newtheorem{conj}[thm]{Conjecture}

\theoremstyle{definition}
\newtheorem{defn}[thm]{Definition}

\newtheorem{exmp}[thm]{Example}

\theoremstyle{remark}
\newtheorem{rem}[thm]{Remark}

\usepackage{color}

\newcommand{\lb}{{[\![}}
\newcommand{\rb}{{]\!]}}

\define\ord{\mathrm{ord}}

\define\Eis{\mathrm{Eis}}

\define\cusp{\mathrm{cusp}}

\define\cC{\mathcal{C}}

\define\Ann{\mathrm{Ann}}

\define\bT{\mathbb{T}}

\newcommand{\Tm}{\bT_{\m}}
\newcommand{\Tmo}{\bT^0_{\m}}
\newcommand{\Tmv}[1]{\bT_{\m,#1}}
\newcommand{\Tmvo}[1]{\bT^0_{\m,#1}}
\newcommand{\Tmtv}[1]{\bT^{\rm tame}_{\m,#1}}
\newcommand{\Tmtvo}[1]{\bT^{\rm 0,tame}_{\m,#1}}
\newcommand{\ds}{\displaystyle}
\newcommand{\Zpx}{\Z_p^\times}
\newcommand{\md}{\rm mod}

\newcommand{\aL}{\mathcal L}
\newcommand{\al}{L}
\newcommand{\alcv}[1]{L^{0}_{#1}}
\newcommand{\alc}{L^{0}}
\renewcommand{\L}{\Lambda}
\renewcommand{\Im}{\I_{\m}}
\newcommand{\Imk}{\I_{\m,k}}

\makeatletter
\let\c@equation\c@thm
\makeatother
\numberwithin{equation}{section}

\title{Iwasawa invariants in residually reducible Hida families}
\author{Robert Pollack}
\address{University of Arizona}
\email{rpollack@arizona.edu}	
	
\author{Preston Wake}
\address{Michigan State University}
\email{wakepres@msu.edu}

\date{}

\begin{document}

\maketitle

\begin{abstract}
We study the variation of $\mu$-invariants of modular forms in a cuspidal Hida family in the case that the family intersects an Eisenstein family.  We allow for intersections that occur because of ``trivial zeros'' (that is, because $p$ divides an Euler factor) as in Mazur's Eisenstein ideal paper, and pay special attention to the case of the $5$-adic family passing through the elliptic curve $X_0(11)$.
\end{abstract}

\section{Introduction}

\subsection{Congruences from $L$-values} In \cite{BP2018}, the first author and Bella\"iche studied the $\mu$-invariants of Hida families of cuspidal eigenforms which admit congruences with Eisenstein series.  The congruences studied there were ones that arose from $p$-divisibilities of $L$-values, in this case, Bernouli numbers.  For instance, consider the case of tame level $N=1$ where $p \mid B_k$, the $k$-th Bernoulli number.  In this situation, there is a cuspidal Hida family which has non-trivial intersections with an Eisenstein Hida family with the intersections occuring at the $p$-adic weights which are zeroes of the $p$-adic $\zeta$-function (which is non-trivial as $p \mid B_k$).  

Assuming certain Hecke algebras were Gorenstein, it was shown in \cite{BP2018} that the $\mu$-invariants in the cuspidal family blew up as one approached these intersection points.  Moreover, precise formulas were conjectured about the values of these $\mu$-invariants which in favorable situations ({\it i.e.}\ when the cuspidal Hida family was rank 1 over weight space) were given simply as the $p$-adic valuation of certain special values of the $p$-adic $\zeta$-function.  When $U_p-1$ generates the Eisenstein ideal, these conjectures were shown to hold for the branch of the $\L$-function with trivial tame character, and the $\lambda$-invariants in these families were shown to be identically zero.

\subsection{Congruences from Euler factors}
In this paper, we aim to treat the analogous situation where congruences arise because of divisibility of Euler factors.  For example, consider the case of $N$ a prime such that $N \equiv 1 \pmod{p}$.  By Mazur's famous result \cite{mazur1978}, as long as $p > 3$, there is always a cuspidal eigenform in $S_2(\Gamma_0(N))$ congruent to the unique ordinary Eisenstein series of weight 2 and level $N$.  The most famous example of this congruence is when $p=5$ and $N=11$.  In this case, there is a unique such cuspidal eigenform and it is exactly the modular form corresponding to the elliptic curve $X_0(11)$.  

The ordinary $p$-stabilization of this form to level $Np = 55$ lives in a Hida family that is rank 1 over weight space.  However, there is a key difference in this situation from the case treated in \cite{BP2018} where there was a unique Eisenstein family to consider.  In our situation, there are two Eisenstein families; namely, one where $U_{N} = 1$ for all weights and one where $U_{N}=N^{k-1}$.  Since $N \equiv 1 \pmod{p}$, these two families are congruent modulo $p$.  Moreover, the localization of the Hida Hecke algebra which corresponds to these two families together with the cuspidal family is not Gorenstein and the methods of \cite{BP2018} do not apply.

To avoid this problem, we employ a method already used in \cite{ohta2014} and \cite{PG5} where we remove $U_N$ from our Hecke algebra and replace it with $w_N$, the Atkin-Lehner involution.  While the $U_N$-eigenvalues of the two Eisenstein families are congruent, the $w_N$-eigenvalues of these families are $1$ and $-1$, and the congruence is broken.  (It is verified in Appendix \ref{appendix:Hida} that Hida theory still works as expected for these modified Hida Hecke algebras.)

We are interested in the Eisenstein family with $w_N$-eigenvalue $-1$ as our cuspidal family has $w_N$-eigenvalue equal to $-1$.  The localization of the modified Hecke algebra corresponding to this Eisenstein family together with our cuspidal family turns out to be Gorenstein and we are good shape to generalize the methods of \cite{BP2018}. 

We state our results in the specific case of $p=5$ and $N=11$.  To set up a little bit of notation, let $L_p^+(k,s)$ denote the (plus) two-variable $p$-adic $L$-function over $\Qinf$, the cyclotomic $\Z_p$-extension, corresponding to the Hida family through $X_0(11)$.   

\begin{thm}
\label{thm:55}
Let $N=11$, $p=5$, and let $k$ be an integer with $k \equiv 2 \pmod{4}$.  Let $f_k$ denote the unique ordinary eigenform in $S_k(\Gamma_0(Np))$ whose residual representation is $1 \oplus \omega$, where $\omega$ is the mod-$p$ cyclotomic character.  Then
$$
L_p^+(k,s) = (N^{k/2}-1) \cdot U(k,s)
$$
where $U(k,s)$ is a unit power series in $k$ and $s$.  In particular,  the Iwasawa invariants of $f_k$ are given by
\begin{enumerate}
\item $\mu(f_k) = \mathrm{val}_p(a_p(f_k)-1) = 1 + \mathrm{val}_p(k)$ and 
\item $\lambda(f_k) = 0$.
\end{enumerate}
\end{thm}

The theorem implies that as one moves closer and closer to weight $0$ in this Hida family, the $\mu$-invariants blow up linearly.   Theorem \ref{thm:55} is a special case of Theorem \ref{thm:general} below.

\subsection{The more general case}
We now discuss the key ingredients that go into Theorem \ref{thm:55} to unravel what is special about this case with $(N,p)=(11,5)$ and weight $2$, and what can be said in greater generality.   To this end, we now let $N$ and $p$ be prime numbers with $N \equiv 1 \pmod{p}$ and $p \ge 5$, and let $k_0$ be an even integer with $0<k_0<p-1$.
Let $\Tm$ denote the completion of the Hida Hecke algebra with tame level $N$ at the maximal ideal given by the residual representation $\overline{\rho} = 1 \oplus \omega^{k_0-1}$ and with Atkin--Lehner sign $w_N=-1$; let $\Tmo$ denote the quotient of $\Tm$ that acts faithfully on cuspforms.  Let $\Im \subseteq \Tm$ denote the Eisenstein ideal.

The following facts hold in the setting of Theorem \ref{thm:55}, where $p=5$, $N=11$, and $k_0=2$:
\begin{enumerate}
\item \label{part:goren}$\Tm$ and $\Tmo$ are Gorenstein
\item \label{part:Up-1}$\Im$ is generated by $U_p-1$
\item \label{part:bern}$p \nmid \frac{B_{k_0}}{k_0}$
\item \label{part:cong}$N \equiv 1 \pmod{p}$ but $N \not \equiv 1 \pmod{p^2}$
\item \label{part:rank}the rank of $\Tmo$ over the Iwasawa algebra $\Lambda$ equals 1
\end{enumerate}
We will now examine each of these facts in turn and how (if at all) they are used in the proof of Theorem \ref{thm:55}; then we will state the general version Theorem \ref{thm:general} of the theorem.
We first note that \eqref{part:Up-1} implies \eqref{part:goren} (see Lemma \ref{lem: both goren}).  In fact,  \eqref{part:goren} holds whenever $\Im$ is principal.

Condition \eqref{part:Up-1} is crucial to obtain as precise a result as in Theorem \ref{thm:55}.  However, we make use of several recent results \cite{Deo,Wake-JEMS,PG5} about the structure of these Hecke algebras which provide a numerical criterion that is equivalent to \eqref{part:Up-1}, as we now explain. 
Let $\log_N : \F_N^\times \to \F_p$ be a surjective homomorphism (noting that $p \mid( N-1)$, this is a reduction modulo $p$ of a choice of discrete logarithm). Further, let $\zeta$ denote a primitive $p$-th root of unity in $\F_N^\times$.
As recalled in Section \ref{sec:goren} below,  the results of \cite{Deo,Wake-JEMS,PG5} imply that if
\begin{equation}
\label{eq:log 1-zeta sum}
\sum_{i=1}^{p-1} i^{k_0-2} \log_N(1-\zeta^i) \neq 0 \text{~in~} \F_p,
\end{equation}
then $\Im$ is a principal ideal, and moreover, if further
$$
p \mbox{ is not a } p \mbox{-th power modulo }N,
$$
then $U_p-1$ generates $\Im$.  Note that when $k_0 \equiv 2 \pmod{p-1}$,  the sum appearing in the first condition $\sum_{i=1}^{p-1} i^{k_0-2} \log_N(1-\zeta^i)$ simplifies to $\log_N(p)$,  so the first condition is equivalent to the second condition.  In particular, in the case $(N,p,k_0)=(11,5,2)$, the fact that 5 is not a 5-th power modulo 11 implies \eqref{part:Up-1} (and thus \eqref{part:goren}) holds.

Moving on, condition \eqref{part:bern} is necessary to ensure that only Euler-factor-type congruences appear. If \eqref{part:bern} fails, then there will be a mix of congruences from $L$-values (as in \cite{BP2018}) and from Euler factors and the situation is more complicated. Although it seems very interesting to understand this situation, it is outside the scope of this paper.

Condition \eqref{part:cong} can be relaxed to $N \equiv 1 \pmod{p}$, but this slightly changes the formula for $\mu(f_k)$ given in Theorem \ref{thm:55}. In this generality, the formula should be $\mu(f_k) = \mathrm{val}_p(N^{k/2}-1)$,  and
$$
\mathrm{val}_p(N^{k/2}-1) = \mathrm{val}_p(N-1) + \mathrm{val}_p(k).
$$ 
When $N \not \equiv 1 \pmod{p^2}$,  this gives the formula $\mu(f_k) = 1+ \mathrm{val}_p(k)$ found in Theorem \ref{thm:55}.

Condition \eqref{part:rank} is important in that without it one cannot simply speak of $f_k$ as there is no longer a unique cuspidal eigenform in our Hida family in each weight.  Further, without this condition, the crossing points of the Eisenstein family and the cuspidal family cannot be as finely controlled.  Nonetheless, we still have a two-variable $p$-adic $L$-function which simply equals $U_p-1$ up to a unit.  Thus we can obtain a formula for the $\mu$-invariants in terms of the $U_p$-eigenvalues of each form, but we can only give a formula that only depends on $k$ and $N$ for the sum of the $\mu$-invariants of all the forms in weight $k$ in our Hida family.

We note that the same results on the structure of Hecke algebras \cite{Deo,Wake-JEMS,PG5,PG3} provide a numerical condition that forces condition \eqref{part:rank} to hold.  (see \cite[Corollary C]{Deo}).  Namely,  assuming \eqref{eq:log 1-zeta sum} and that $p$ satisfies Vandiver's conjecture, $\Tmo$ has rank 1 over $\Lambda$ if and only if
$$
\prod_{i=1}^{N-1} i^{\left(\sum_{j=1}^{i-1} j^{k_0-1}\right)} \mbox{ is not a } p \mbox{-th power modulo }N
$$
(see \cite[Corollary C]{Deo}).

For a particular weight $k$,  consider the condition that the specialization $\bT^0_{\m,k}$ of $\bT^0_\m$ to weight $k$ is a domain.  This is more general than the condition that $\Tmo$ have rank 1 over $\Lambda$, since $\bT^0_{\m,k}=\Z_p$ in that case.  In other words,  the assumption that $\Tmo$ has rank 1 over $\Lambda$ implies that there be a single cuspform congruent to the Eisenstein series in weight $k$, whereas the assumption that $\bT^0_{\m,k}$ is a domain merely implies that there is a single local-Galois-conjugacy class of such forms.
When $\bT^0_{\m,k}$ is a domain,  the fact that the conjugate forms have the same Iwasawa invariants allows us completely control the situation.  On the one hand, when $\bT^0_{\m,k}$ is not equal to $\Z_p$, we know of no numerical criterion that forces it to be a domain.  But on the other hand, examples where it is not a domain seem rare.   Using Sage \cite{sagemath} we computed the following:\ when $p=5$ and $k_0=2$, there are 163 primes $N$ less than 5000 for which $N \equiv 1 \pmod{5}$.  Of these, 48 have $\Tmo$ with rank greater than 1, but only 6 have more than one Galois conjugacy class of forms in weight 2.

We now state our main result in more general terms.  For this, we require some notation that will be discussed in more detail in Section \ref{sec:normalizations}.
Let $L^+_p(\m)$ denote the (plus) two-variable $p$-adic $L$-function attached to our Hida family,  and write $L^+_p(\m,\omega^0)$ for its branch corresponding to $\Qinf$, the cyclotomic $\Z_p$-extension---that is, the branch with trivial tame character.
For an eigenform $g$ in our Hida family, write $\varpi_g$ for a uniformizer of the normalization of the ring generated by the Hecke eigenvalues of $g$. 

\begin{thm}
\label{thm:general}
Assume that
\begin{enumerate}
\item $p \nmid \frac{B_{k_0}}{k_0}$, 
\item $\ds \sum_{i=1}^{p-1} i^{k_0-2} \log_N(1-\zeta^i) \neq 0$ in $\F_p$, and
\item $p$ is not a $p$-th power modulo $N$.
\end{enumerate}
Then $L^+_p(\m,\omega^0)$ has the simple form
$$
L^+_p(\m,\omega^0) = (U_p-1) \cdot U
$$
where $U$ is a unit. In particular, for every cuspform $g$ that corresponds to a height-one prime of $\Tmo$,
$$
\mu(g) = \ord_{\varpi_g}(a_p(g)-1) \mbox{ and }
\lambda(g) = 0.
$$
Further, for every integer $k$ with $k \equiv k_0 \pmod{p-1}$,
$$
\sum_g \mu(g) = \mathrm{val}_p(N-1) + \mathrm{val}_p(k)
$$
where the sum is over all cuspfoms $g$ that correspond to height-one primes of $\bT_{\m,k}$.
\end{thm}
Theorem \ref{thm:general} is stated, and proven, in a more precise form in Theorem \ref{thm:Mazur U_p-1 gens} below.

\subsection{Normalizations}
\label{sec:normalizations}

Like the results in \cite{BP2018}, the above theorem can be thought of as explaining  $\mu$-invariants through $p$-adic variation.  For any individual form, a positive $\mu$-invariant can almost be thought of as an error in normalization.  One can simply change the complex period defining the $p$-adic $L$-function by a power of $p$ to simply force a $\mu$-invariant to be 0.  However, in the family one now sees the $\mu$-invariants as arising from valuations of special values of $p$-adic analytic functions.  In the setting of \cite{BP2018}, in the rank 1 case, the relevant analytic function was the $p$-adic $\zeta$-function (which itself has $\mu$-invariant 0).  In this paper, the relevant function is simply $\langle N \rangle ^{1/2}-1$ where $\langle N \rangle$ is the element in $\Lambda$ which specializes to $N^k$ in weight $k$.  Note that $\langle N \rangle ^{1/2}-1$ also has $\mu$-invariant 0.

With that said, one can wonder why these two-variable $p$-adic $L$-functions are divisible by these analytic functions that depends only on $k$ and not on the cyclotomic variable $s$.  Indeed, the origin of this project was an attempt to reconcile the results of \cite{BP2018} with results of the second author which said that these $\mu$-invariants were identically 0 along the family!

The reconciliation of these results is that there are two natural normalizations of the $p$-adic $L$-function and one leads to positive $\mu$-invariants which blow up in the family and the other leads to $\mu$-invariants that are 0 everywhere.  We now explain.

In the construction, due to Mazur and Kitigawa, of the two-variable $p$-adic $L$-function, one considers a certain explicit class
$$
\cL^+_p(\m) \in X\lb\Zpx\rb := \varprojlim_n H^1(Y_1(Np^r),\Z_p)^{-,\ord}_{\m}\lb\Zpx\rb
$$
built out of modular symbols where the upper ``$\ord$" denotes the ordinary projector and $\m$ denotes the completion at the maximal ideal corresponding to our Hida family (the change in sign in the notation is due to Poincar\'e duality:\ we think of the negative-signed cohomology class $\cL^+_p(\m)$ as a \emph{functional} on modular symbols with positive sign \S \ref{sec:pL}).  
We refer to this class as the two-variable $p$-adic $L$-{\it symbol}.  

Typically, when constructing $p$-adic $L$-functions,  one extracts a function from a symbol by choosing a basis for the space of symbols.  In this case, that would involve choosing a basis for the space $X$ as a $\Tm$-module. However, the space $X$ is not necessarily cyclic as a $\Tm$-module; in fact, by Ohta's Eichler--Shimura isomorphism for Hida families \cite{ohta1999}, it is known that $X$ is a dualizing module for $\Tm$.  Hence $X$ is free of rank $1$ over $\Tm$ if and only if $\Tm$ is a Gorenstein ring.  In fact, if $\Tm$ is Gorenstein, then there is a canonical generator $\zinf \in X$ and there is a unique $L_p(\m) \in \Tm \lb \Zpx\rb$ such that
$$
\cL^+_p(\m) = L^+_p(\m) \cdot \{0, \infty\}.
$$
The element $L^+_p(\m)$ can be viewed as a two-variable $p$-adic $L$-function.  Indeed, the cyclotomic variable comes from the fact that $L^+_p(\m)$ lives in a group algebra over $\Zpx$ while the weight variable is parametrized by $\Spec(\Tm)$. If $f$ is a classical form in our Hida family and $\p_f$ is the corresponding height 1 prime ideal of $\Tm$, then the image of $L^+_p(\m)$ in $\Tm / \p_f\lb\Zpx\rb$ is the one-variable $p$-adic $L$-function of $f$.  

As discussed in \cite[Section 1.2]{BP2018}, the specializations of $L^+_p(\m)$ to Eisenstein forms in the family always vanish and this is the reason why the $\mu$-invariants blow up in the family.  We write $L_p^{+,\md}(\m)$ for the image of $L^+_p(\m)$ in $\Tmo\lb\Zpx\rb$ and we refer to $L_p^{+,\md}(\m)$ as the two-variable $p$-adic $L$-function with the {\it modular} normalization.  Here ``modular" is referring to the fact that this $L$-function first arose from a Hecke-algebra on the full space of modular forms.

On the other hand there is another natural normalization.  A closer examination of the definition of the $p$-adic $L$-symbol shows that it actually lives in the cohomology of the compact modular curve:
$$
\cL^+_p(\m) \in X^0\lb\Zpx\rb := \varprojlim_n H^1(X_1(Np^r),\Z_p)^{-,\ord}_{\m}\lb\Zpx\rb.
$$
If we instead assume that $\Tmo$, the cuspidal Hecke algebra localized at $\m$, is Gorenstein, then we again have $X^0$ is free of rank 1 over $\Tmo$; however, in this case there is no natural generator of this space.  Nonetheless, pick a generator $e$ of $X^0$, and write
$$
\cL^+_p(\m) = L_p^{+,\cusp}(\m) \cdot e
$$
for a unique $L_p^{+,\cusp}(\m) \in \Tmo\lb\Zpx\rb$.  We call this a two-variable $p$-adic $L$-function with the {\it cuspidal} normalization as it was defined directly from the cuspidal Hecke algebra.  It is well-defined up to a unit in $\Tmo$.

When both $\Tm$ and $\Tmo$ are Gorenstein, then both definitions $L_p^{+,\md}(\m)$ and $L_p^{+,\cusp}(\m) $ make sense and one can ask how they are related.  Following ideas of Ohta \cite{ohta2005}, we show that both $\Tm$ and $\Tmo$ are Gorenstein if and only if the Eisenstein ideal $\Im$ is principal. Moreover, we show that $L_p^{+,\cusp}(\m)$ divides $L_p^{+,\md}(\m)$ and that the ratio is a generator of the Eisenstein ideal. 
 Note that in Theorem \ref{thm:general}, $U_p-1$ generates the Eisenstein ideal and this is exactly the analytic factor that causes the $\mu$-invariants to grow in the family.  In that theorem, the modular normalization is used.  If we instead used the cuspidal normalization, we would lose that factor of $U_p-1$ and all of the $\mu$-invariants would be 0.  

\subsection{Further questions}
In this paper, we only consider that analytic side of Iwasawa theory, in that we only consider $p$-adic $L$-functions.  There is a parallel algebraic side,  involving Selmer groups, that appears to be closely related to our results.  For instance,  specializing to weight 2 in Theorem \ref{thm:55} gives that the analytic $\mu$-invariant of $X_0(11)$ is $1$; this corresponds to the famous fact that the algebraic $\mu$-invariant of $X_0(11)$ is $1$ \cite[Section 10]{mazur1972}.
This algebraic side is particularly interesting in view of conjectures of Greenberg \cite{greenberg2001} regarding algebraic $\mu$-invariants of elliptic curves.  There is also an algebraic analog of the normalizations discussed in Section \ref{sec:normalizations}. The Selmer group depends on a choice of $\Z_p$-lattice in the Galois representation and there are (at least) two choices of  
lattice that come from geometry:\ a cuspidal lattice, coming from $H^1(X_1(Np^r),\Z_p)$, and a modular lattice, coming from $H^1(Y_1(Np^r),\Z_p)$. We are currently working to prove the algebraic analog of Theorem~\ref{thm:general},  and to prove a kind of Iwasawa main conjecture for $\mu$-invariants:\ that the algebraic $\mu$-invariant for the cuspidal lattice matches the analytic $\mu$-invariant for the cuspidal normalization, and similiarly for the modular normalization.

We also restrict our attention to the case where the tame level $N$ is prime and where $p \nmid B_{k_0}$. This restriction on the level is natural in view of our goal to understand the context for the example of $X_0(11)$.  However,  this assumption has the drawback of limiting the types of examples we can explore:\ for instance, $X_0(11)$ is the only elliptic curve we know of for which Theorem \ref{thm:general} applies.  Many of the techniques developed in this paper are general and can be applied to other situations as soon as the relevant results on the structure of Hecke algebras are known. Some interesting situations to consider include:
\begin{itemize}
\item  $N$ is squarefree. (There are some results on the structure of Hecke algebras proven in \cite{PG5}.)
\item $N$ is the square of a prime. (This situation was considered in \cite{LW}.)
\item  $N$ is prime but $p \mid B_{k_0}$, so there are congruences of both ``$L$-value'' and ``Euler-factor'' type. 
\end{itemize}
In this greater generality,  it will usually not be true that the Eisenstein ideal is principal, so it is unlikely that one can hope for results as precise as Theorem \ref{thm:general}.  We make some conjectures (see Conjectures \ref{conj:mu conj} and \ref{conj:Mazur mu conj}) regarding what is happening with Iwasawa invariants in some cases where the Eisenstein ideal fails to be principal,  but where one of $\Tm$ and $\Tmo$ is Gorenstein.  Examples statisfying these conditions in our current setting are rare, so we were unable to test our conjectures numerically. However,  there are many such examples in the more relaxed situations above (for instance, \cite[Example 1.9.3]{PG5} is an example with small squarefree level and weight $2$ where $\Im$ is not principal, $\Tm$ is not Gorenstein, but $\Tmo$ is Gorenstein). We hope that these conjectures can serve as a starting point for investigating these more general situations.

\subsection{Structure of paper}
In Section \ref{sec:Axiomatics}, we begin with an axiomatic approach to describing Hecke algebras, Eisenstein ideals, and $p$-adic $L$-functions when there is a unique Eisenstein series parametrized by our Hecke algebra.  This axiomatic approach applies equally to the case where the congruences arise by $p$-divisibility of $L$-values or of Euler factors and studies the two possible normalizations discussed in \S \ref{sec:normalizations}.  In Section \ref{sec:pL}, we review the construction of the two-variable $p$-adic $L$-symbol $\cL^+_p(\m)$. In Section \ref{sec:primitive}, we consider the setting of \cite{BP2018}.  We reprove the main results of \cite{BP2018} as an illustration of the axiomatics of Section \ref{sec:Axiomatics} and also make conjectures about what happens when the Eisenstein ideal is not principal.  In Section \ref{sec:mazur}, we move to the case of trivial tame character and prime level.
This section contains our main results on analytic $\mu$-invariants.  
Finally, in Appendix \ref{appendix:Hida}, we verify that Hida theory works equally well for Hecke algebras in which $U_N$ has been replaced by $w_N$.

\vspace{0.2cm}

\noindent
{\it Acknowledgements}:\
We dedicate this paper to the memory of Jo\"el Bella\"iche. 
The impact of Jo\"el's mathematics on this paper,  from his pioneering work on pseudorepresentations to the paper \cite{BP2018}, will be clear to the reader.  Even the genesis of this paper goes back to a meeting between the authors and Jo\"el, about \cite{BP2018} and generalizations,  at the Glenn Stevens birthday conference in 2014.  At the same conference,  Jo\"el first introduced the second author to Carl Wang-Erickson,  which led to a fruitful collaboration that provided many of the results that make this paper possible,  especially  \cite{PG3,PG4,PG5}.  

We would like to thank Carl Wang-Erickson for many helpful discussions on the topic of this paper.

The first author thanks the Max Planck Institute for Mathematics in Bonn and the second author thanks the Institute for Advanced Study in Princeton for wonderful environments during their stays there. 
The first author was supported by NSF grants DMS-1702178 and DMS-2302285 and the second author by NSF grant DMS-1901867.

\section{Axiomatics}
\label{sec:Axiomatics}
In Section \ref{sec:primitive}, we consider the same situation as \cite{BP2018}, with a primitive tame character, and in Section \ref{sec:mazur}, we consider the non-primitive setting of Mazur's Eisenstein ideal. The two situations have a lot of commonalities, and to emphasize this,  in this section, we consider a purely abstract situation which covers both of these at once. 

The main point is that the standard construction of the two-variable $p$-adic $L$-function gives not a \emph{function} but an element of a two-variable modular symbols space. In both situations, the modular symbols space is known to be isomorphic to the \emph{dualizing module} of the cuspidal Hecke algebra.  If this dualizing module is isomorphic to the Hecke algebra ({\it i.e.}\ if the cuspidal Hecke algebra is Gorenstein), then the symbol gives a function.  The $p$-adic $L$-function can also be thought of as an element of the larger dualizing module of the full Hecke algebra, and if that Hecke algebra is Gorenstein, then the symbol again gives a function. If both Hecke algebras are Gorenstein, this gives two different ways to produce a function from the symbol, and the two functions are different by a normalizing factor.

Much of this discussion revolves around commutative algebra questions of whether one or the other, or both, or these Hecke algebras is Gorenstein. In this section, we discuss these commutative algebra issues in the abstract.

\subsection{Setup}
The first key property common to both of the situations we consider is that, though there are congruences between cuspidal families and an Eisenstein family, there are no congruences between different Eisenstein families.  We axiomatize this as follows. 

\begin{defn}
\label{def:single eisen}
A \emph{single-Eisenstein Hecke algebra setup} is the data of $(\sO,A,R,\E, r_0)$ where
\begin{itemize}
\item $\sO$ is a complete discrete valuation ring with uniformizer $\varpi$ and residue field $k$,
\item $A$ is a flat,  local $\sO$-algebra that is a complete commutative Noetherian local complete intersection ring with maximal ideal $\m_A$,
\item $R$ is a commutative local $A$-algebra,
\item $\E: R \to A$ is an $A$-algebra homomorphism, and
\item $r_0 \in R$ is an element the annihilates $\ker(\E)$
\end{itemize}
satisfying the conditions
\begin{itemize}
\item $R$ is a free $A$-module of finite rank $d$,
\item $R/\Ann_R(\ker(\E))$ is a free $A$-module of rank $d-1$,
\item $r_0$ generates $\Ann_R(\ker(\E))$ as an $A$-module, and
\item $\E(r_0) \in A$ is a not a zero divisor.
\end{itemize}
Given this setup, let $I=\ker(\E)$, let $R^0=R/\Ann_R(I)$, and let $I^0 \subset R^0$  be the image of $I$ in $R^0$.
\end{defn}

\begin{rem}
In the situations we consider,  the data $(\sO,A,R,\E, r_0)$ are as follows:
\begin{itemize}
\item $\sO$ is $\Z_p$ or the valuation ring of a finite extension of $\Q_p$,
\item $A$ is a ring of diamond operators,
\item $R$ is a localization of a Hecke algebra at an Eisenstein maximal ideal, 
\item $\E$ is the action of $R$ on an $A$-family of Eisenstein series, and
\item $r_0$ is Hecke operator with the property that, for a modular form $f$, the coefficient $a_1(r_0 f)$ of $q$ in the $q$-expansion of $r_0f$ equals the residue of $f$ at a cusp.
\end{itemize}
If $\E$ is the only family of Eisenstein series supported by $R$, then $I$ is the Eisenstein ideal and $R^0$ is maximal the quotient of $R$ that acts faithfully on cuspforms.
\end{rem}

Let $(\sO,A,R,\E,r_0)$ be a single-Eisenstein Hecke algebra setup. Note that the map $\E$ induces an isomorphism $R/I \cong A$, making $A$ into an $R$-algebra.  Moreover, the map $A \to \Ann_R(I)$ given by $1 \mapsto r_0$ is an isomorphism of $R$-modules, so that for all $r \in R$,
\begin{equation}
\label{eq:action on r_0}
r \cdot r_0 = \E(r) r_0.
\end{equation}
Note also that the action of $R$ on $I$ factors through $R^0$, so that if $r \in R^0$ and $y \in I$, then $ry$ is a well-defined element of $I$.

The second key property is that the $p$-adic $L$-function is given as an element of a module that is dual to the cuspidal Hecke algebra.

\begin{defn}
\label{def:Lsetup}
Let $(\sO,A,R,\E,r_0)$ be a single-Eisenstein Hecke algebra setup. Then an \emph{$L$-symbol setup} is the data of a triple $(X,\phi,x,\cL)$ where
\begin{itemize}
\item $X$ is a $R$-module,
\item $\phi: X \to A$ is an $A$-linear functional on $X$,  with kernel $X^0=\ker(\phi)$,
\item $x \in X$ is an element such that $\phi(x)=1$, and
\item $\cL \in \ker(\phi) \otimes_\sO  \sO \lb \Z_p^\times\rb$,
\end{itemize}
satisfying the condition that 
\begin{itemize}
\item there is an isomorphism $f: X \isoto \Hom_A(R,A)$ of $R$-modules sending $X^0$ isomorphically to $\Hom_A(R^0,A)$.
\end{itemize}
Note that, for all $g \in \Hom_A(R,A)$ and $r\in R$,  the element $(r-\E(r))g$ is in $\Hom_A(R^0,A)$ by \eqref{eq:action on r_0}.  This implies that, for all $y \in X$, the element $(r-\E(r))y$ is in $X^0$, so that $\phi$ must be a morphism of $R$-modules:
\[
\phi(ry)=\E(r)\phi(y).
\]
\end{defn}

Recall that a local ring is called \emph{Gorenstein} if it has a dualizing module that is free of rank $1$ (see \cite[Section 3]{BH1993} for a discussion of Gorenstein rings, and note that the authors use the term \emph{canonical module} for what we call \emph{dualizing module}).  For our purposes, the important facts to remember are that local complete intersection rings are Gorenstein \cite[Proposition 3.1.20, pg.~96]{BH1993} and that if $B \to B'$ is a finite flat local ring homomorphism and $M$ is a dualizing module for $B$, then $\Hom_B(B',M)$ is a dualizing module for $B'$ \cite[Theorem 3.3.7, pg.~112]{BH1993}.  In particular, in an $L$-symbol setup, $X$ is a dualizing module for $R$ and $X^0$ is a dualizing module for $R^0$.
\subsection{Algebra around Gorensteinness}
Let  $(\sO,A,R,\E,r_0)$ be a single-Eisenstein Hecke algebra setup.  The composite map
\[
R \xrightarrow{\E} A \to A/\E(r_0)A
\]
factors through $R^0$. Let $A^0=A/\E(r_0)A$ and let $\E^0:R^0 \to A^0$ denote the induced map.  Since the two quotient maps $R \to R^0 \xrightarrow{\E^0} A^0$ and $R \xrightarrow{\E} A \to A^0$ are equal, there is a map to the fiber product of rings:
\[
R \to R^0 \times_{A^0} A = \{(r,a) \in R^0 \times A \ | \ \E^0(r) \equiv a \mod{\E(r_0)} \}.
\]

\begin{lem}
\label{lem:fiber prod}
The map $R \to R^0 \times_{A^0} A$ is an isomorphism of $A$-algebras. In particular, the natural maps $R^0/I^0 \to A^0$ and $I \to I^0$ are isomorphisms of $R$-modules.
\end{lem}
\begin{proof}
Consider the commutative diagram of $R$-modules with exact rows
\[
\xymatrix{
0 \ar[r] & Ar_0 \ar[r] \ar[d] & R \ar[r] \ar[d]^-\E & R^0 \ar[r] \ar[d]^-{\E^0} & 0 \\
0 \ar[r] & A\E(r_0) \ar[r] & A   \ar[r] & A^0 \ar[r] & 0.
}\]
The leftmost vertical arrow is an isomorphism: it is surjective by definition and injective since $\E(r_0)$ is assumed to be a non-zero-divisor.  By the five-lemma, the map $\ker(\E) \to \ker(\E^0)$ is an isomorphism.  This is enough to imply that the map $R \to R^0 \times_{A^0} A$ is an isomorphism.  Indeed, the kernel is easily seen to equal $\ker(Ar_0 \to A\E(r_0))$, which is zero.  To see it is surjective, let $(t^0,a) \in R^0 \times_{A^0} A$, and choose a lift $t \in R$ of $t$.  Since $\E(t) \equiv \E^0(t^0) \equiv a \mod{A\E(r_0)}$,  it follows that $\E(t)-a \in A\E(r_0)$, so $\E(t)-a=a'\E(r_0)$ for some $a'\in A$. Then $t -a'r_0$ maps to $(t^0,\E(t-a'r_0))=(t^0,a)$. 
\end{proof}

Now let $(X,\phi,x, \cL)$ be an $L$-symbols setup. The following gives criteria for $R$ to be Gorenstein.

\begin{lem}
\label{lem: T goren}
The following are equivalent:
\begin{enumerate}
\item
\label{part:Rgoren1}
 $R$ is Gorenstein,
\item 
\label{part:Rgoren2}
 $X$ is a free $R$-module of rank 1,
\item 
\label{part:Rgoren3}
 $X$ is cyclic as an $R$-module,
\item
\label{part:Rgoren4}
  $X$ is generated by $x$ as an $R$-module,
\item 
\label{part:Rgoren5}
 The map $I \to X^0$ given by $t \mapsto t \cdot x$ is an isomorphism of $R$-modules.
\end{enumerate}
\end{lem}
\begin{proof}
The equivalence of \eqref{part:Rgoren1} and \eqref{part:Rgoren2} is the definition of Gorenstein, and clearly \eqref{part:Rgoren2} implies \eqref{part:Rgoren3}.  Assume \eqref{part:Rgoren3}, and let $y \in X$ be a generator of $X$ as an $R$-module.  Let $r \in R$ be such that $ry=x$.  Then 
\[
1=\phi(x)=\phi(ry)=\E(r)\phi(y).
\]
This implies $\E(r) \in A^\times$, so, since $R$ is local, $r \in R^\times$.  Since $x=ry$, this implies \eqref{part:Rgoren4}. 
Since $X$ is a faithful $R$-module, \eqref{part:Rgoren4} implies \eqref{part:Rgoren2}.
Lastly, the equivalence of \eqref{part:Rgoren4} and \eqref{part:Rgoren5} follows by applying the five-lemma to the commutative diagram
\begin{equation}
\label{eq:R and X diagram}
\xymatrix{
0 \ar[r] & I \ar[r] \ar[d] & R \ar[r]^\E \ar[d]^{r \mapsto r \cdot x} & A\ar[r] \ar@{=}[d] & 0 \\
0 \ar[r] & X^0 \ar[r] & X \ar[r]^\phi & A \ar[r] & 0.}
\end{equation}
\end{proof}

It follows from the lemma that if $R$ is Gorenstein, then $I$ is a dualizing module for $R^0$.  In particular, if $R$ and $R^0$ are both Gorenstein, then $I$ is principal.  We will see that the converse is true as well. First, we require the following lemma on the structure of $R$ when $I$ is principal.

\begin{lem}
\label{lem:structure of T}
Assume that $I$ is a principal ideal with a generator $t \in I$. Let $F(X) \in A[X]$ be the characteristic polynomial of the $A$-linear endomorphism
\[
R^0 \to R^0, \ r \mapsto tr.
\]
Then there are $A$-algebra isomorphisms
\[
A[X]/(F(X)) \isoto R^0, \ A[X]/(XF(X)) \isoto R
\]
given by $X \mapsto t$. Moreover, $F(X)$ is a distinguished polynomial with $F(0)A=\E(r_0)A$.
\end{lem}
\begin{proof}
Let $\phi$ denote the $A$-algebra homomorphism
\[
A[X] \to R, \ X \mapsto t.
\]
First note that $\phi$ is surjective. To see this, it suffices, by Nakayama's lemma to prove that $\phi \otimes_A A/\m_A$ is surjective, so we may assume that $A$ is a field. In that case, $R$ is a $d$-dimensional algebra over the field $A$ with maximal ideal $I$, so $I^d=0$. Fix $r \in R$ and note that $r-\E(r) \in I$, so there is $r_1 \in R$ such that $r=\E(r)+r_1t$. Let $r_0=R$ and inductively choose $r_i \in R$ such that
\[
r_i=\E(r_i)+r_{i+1}t.
\]
Then
\[
r = \E(r_0)+\E(r_1)t+\dots+\E(r_{d-1})t^{d-1},
\]
so $r$ is in the image of $\phi$.

Composing $\phi$ with the surjection $R \onto R^0$ yields a surjective map $\phi^0:A[X] \to R^0$. Then $\phi^0$ factors through $A[ X ] /(F(X))$ by the Cayley-Hamilton Theorem, and, since $F(X)$ is monic, $A[ X ]/(F(X))$ is a free $A$-module of rank equal to the degree of $F(X)$. Since $R^0$ is a free $A$-module of rank equal to $\deg(F)$, this implies that $\phi^0$ is an isomorphism. Since $R^0$ is local, it follows that $F(X)$ is distinguished. The remaining parts follow from the isomorphism $R \cong R^0 \times_{A^0} A$ of Lemma \ref{lem:fiber prod}.
\end{proof}

\begin{exmp}
\label{exmp: rank 1} 
Suppose that $R^0=A$. Then Lemma \ref{lem:fiber prod} implies that $I^0=\E(r_0)A$ and that $I$ is generated by $r_0-\E(r_0)$. Then there is an isomorphism 
\[
A[X]/(X^2-\E(r_0) X) \isoto R
\]
given by $X \mapsto r_0-\E(r_0)$.
\end{exmp}

\begin{lem}
\label{lem: both goren}
The following are equivalent:
\begin{enumerate}
\item
\label{part:both1}
 $I^0$ is a principal ideal.
\item
\label{part:both2}
  $I$ is a principal ideal.
\item 
\label{part:both3}
 $I$ is a free $R^0$-module of rank $1$.
\item
\label{part:both4}
  Both $R$ and $R^0$ are local complete intersection rings.
\item 
\label{part:both5}
 Both $R$ and $R^0$ are Gorenstein.
\item 
\label{part:both6}
 There is $t \in I$ such that $t \cdot x$ is a generator of $X^0$ as a $R^0$-module.
\end{enumerate}
Furthermore, if all of these statements are true, then \eqref{part:both6} is true for all generators $t$ of $I$ and every such $t$ is a generator of $I$.
\end{lem}
\begin{proof}
By Lemma \ref{lem:fiber prod}, the map $I \to I^0$ is an isomorphism, so \eqref{part:both1} implies \eqref{part:both2}. Since $I$ is a faithful $R^0$-module,  \eqref{part:both2} implies \eqref{part:both3}, and clearly \eqref{part:both3} implies \eqref{part:both1}.

Now assume \eqref{part:both2}.  By Lemma \ref{lem:structure of T},  there are isomorphisms
\[
R^0/\m_AR^0 \cong k[X]/(X^{d-1}) , \ R/\m_A R\cong k[X]/(X^{d}),
\]
so $R^0/\m_AR^0$ and $R/\m_AR$ are visibly local complete intersection rings.  Since $A$ is a local complete intersection and $A \to R^0$ and $A\to R$ are flat, this implies \eqref{part:both4} (see \cite[\href{https://stacks.math.columbia.edu/tag/09Q7}{Tag 09Q7}]{stacks-project}).  Moreover, \eqref{part:both4} implies \eqref{part:both5} by general algebra (see \cite[\href{https://stacks.math.columbia.edu/tag/0DW6}{Tag 0DW6}]{stacks-project}).

Now assume \eqref{part:both5}.  Since $R^0$ is Gorenstein, there is a generator $x^0 \in X^0$ of $X^0$ as an $R^0$-module.  By Lemma \ref{lem: T goren} (part \eqref{part:Rgoren5}) implies that $x^0$ is of the form $x^0=tx$ for some $t \in I$, proving \eqref{part:both6}.

Now assume \eqref{part:both6}. Then the leftmost vertical map in \eqref{eq:R and X diagram} is surjective. By the snake lemma, the center vertical map is also surjective, and, since $X$ is a faithful $R$-module, this implies that the center vertical map (and hence all the vertical maps by the 5 lemma) in \eqref{eq:R and X diagram} are isomorphisms. By \eqref{part:both6}, the composite map
\[
R^0 \xrightarrow{r \mapsto rt} I \xrightarrow{a \mapsto ax} X^0
\]
is surjective. Since the second map is an isomorphism, this implies that $t$ generates $I$, proving \eqref{part:both2}.
\end{proof}

\subsection{Algebra around $L$-functions} 
Let $(\sO,A,R,\E,r_0)$ be a single-Eisenstein Hecke algebra setup and let $(X,\phi,x,\cL)$ be an $L$-symbol setup for it. In this section,  we define different elements of $R\lb \Z_p^\times \rb$,  under different Gorenstein hypotheses, which we think of as normalizations of the $L$-function (as in Section \ref{sec:normalizations}) associated to the $L$-symbol setup. When both $R$ and $R^0$ are Gorenstein, then there are two distinct normalizations and we compare them.

\begin{lem}
\label{lem: L when T goren}
Suppose that $R$ is Gorenstein. Then there is a unique $\al \in R \lb \Zpx \rb$ such that $\aL = \al \cdot x$. Moreover, $\al$ is in the subset $I \lb \Zpx \rb \subset R \lb \Zpx \rb$.
\end{lem}
\begin{proof}
By Lemma \ref{lem: T goren}, the map
\[
R \lb \Zpx \rb \xrightarrow{r \mapsto rx} X \otimes_{\sO}  \sO \lb \Zpx \rb 
\]
is an isomorphism and it induces an isomorphism $I \lb \Zpx \rb \to X^0\otimes_{\sO}  \sO \lb \Zpx \rb$.  Then $L$ is the preimage of $\cL$ under the second isomorphism.
\end{proof}

\begin{defn}
If $R$ is Gorenstein, the element $\al \in R \lb \Zpx \rb$ of Lemma \ref{lem: L when T goren} is called the {\it modular normalization of the $p$-adic $L$-function}. 
\end{defn}

\begin{lem}
\label{lem: L when T^0 goren}
Suppose that $R^0$ is Gorenstein,  so that $X^0$ is a free  $R^0$-module of rank $1$.  Then, for each generator $e$ of $X^0$, there is a unique  $\alcv{e} \in R^0 \lb \Zpx \rb$ such that $\aL = \alcv{e} \cdot e$. Moreover, the class of $\alcv{e}$ in the quotient multiplicitive monoid $R^0\lb \Zpx \rb / (R^0)^\times$ is independent of the choice of generator $e$.
\end{lem}
\begin{proof}
The first part is clear since $X^0 \otimes_\sO \sO\lb \Zpx \rb$ is a free $R^0 \lb \Zpx \rb$-module with generator $e$. If $f$ is another generator for $X^0$, then $f=Ue$ for a unique $U \in (R^0)^\times$,  and
\[
\alcv{e} \cdot e = \aL = \alcv{f} \cdot f =  \alcv{f} \cdot Ue 
\]
so $\alcv{e}=\alcv{f} \cdot U$ and the second part follows.
\end{proof}
\begin{defn}
If $R^0$ is Gorenstein,  let $\alc$ denote the class of $\alcv{e}$ in $R^0\lb \Zpx \rb / (R^0)^\times$ for some choice of $e$ (Lemma \ref{lem: L when T^0 goren} says that $\alc$ is independent of the choice). The class $\alc$ is called the \emph{cuspidal normalization of the $p$-adic $L$-function}.  
\end{defn}

If both $R$ and $R^0$ are Gorenstein, then there are two different normalizations $L$ and $L^0$.  By Lemma \ref{lem: both goren},  both $R$ and $R^0$ are Gorenstein if and only if $I$ is principal; the following lemma essentially says that the two normalizations $L$ and $L^0$ differ by a generator of $I$.

\begin{lem}
\label{lem: L when I principal}
Suppose that $I$ is principal.
\begin{enumerate}
\item There is a unique element $\al \in I \lb \Zpx \rb$ such that $\aL = \al \cdot x$.
\item For each generator $t \in I$, the element $t x \in X^0$ is a generator for $X^0$ as a $R^0$-module, and there is a unique element $\alcv{tx} \in R^0 \lb \Zpx \rb$ such that $\aL = \alcv{tx} \cdot tx$.
\item \label{part:rel} For each generator $t \in I$,  there is an equality  $\al = \alcv{tx} \cdot t$ of elements of~$I\lb \Zpx \rb$.
\end{enumerate}
\end{lem}
\begin{proof}~
\begin{enumerate}
\item By Lemma \ref{lem: both goren}, $R$ is Gorenstein, so (1) follows from Lemma \ref{lem: L when T goren}.
\item By Lemma \ref{lem: both goren}, $t x \in X^0$ is a generator for $X^0$ as a $T^0$-module, and (2) follows from Lemma \ref{lem: L when T^0 goren}.
\item By Lemmas \ref{lem:fiber prod} and \ref{lem: both goren}, the map $R \to I$ given by $1 \mapsto t$ factors through an isomorphism $R^0 \isoto I$. Since $\al \in I \lb \Zpx \rb$, there is a unique element $\al_t \in T^0 \lb \Zpx \rb$ such that $\al = \al_t \cdot t$. Then, by (1), $\cL=\al_t \cdot tx$. On the other hand, by (2), $\alcv{tx}$ is the unique element of $T^0 \lb \Zpx\rb$ satisfying $\aL = \alcv{tx} \cdot tx$. Hence $\al_t = \alcv{tx}$, and (3) follows.
\end{enumerate}\end{proof}

Note that $R\lb \Zpx \rb$ (and similarly $R^0\lb \Zpx \rb$) is a semi-local ring with components labeled by the characters of $(\Z/p\Z)^\times$; 
choosing a generator of $1 + p\Z_p$, each component is isomorphic to a power series ring $R \lb u \rb$.  For $j  \in \{0, \dots, p-2\}$ and $f \in R \lb \Zpx \rb$, let $f(\omega^j) \in R \lb u \rb$ denote the image of $f$ in the component labeled by $\omega^j$; this is called the \emph{$\omega^j$-branch} of $f$. 
If $c: R \lb u \rb \to R$ and is a $R$-algebra homomorphism, we often write $f(\omega^j,c) \in R$ instead of $c(f(\omega^j))$, and think of this as `` evaluation at $c$''.

\begin{lem}
\label{lem: lam(c) generates}
Suppose that $I$ is principal and that there is a $j \in \{0,\dots, p-2\}$ and an $R$-algebra homomorphism $c: R\lb u \rb \onto R$ such that $\al(\omega^j,c) \in I$ is a generator. 
Let $t=\al(\omega^j,c)$ and let $L_{tx}^0 \in R^0\lb \Zpx \rb$ be as in Lemma \ref{lem: L when I principal}(2).  Then:
\begin{enumerate}
\item \label{part:factor}$\al(\omega^j) = \al(\omega^j,c) \cdot L_{tx}^0(\omega^j)$, and
\item \label{part:unit} $L_{tx}^0(\omega^j) \in (R^0 \lb u \rb)^\times$.
\end{enumerate}
In particular, $\alc(\omega^j) \in   (R^0 \lb u \rb)^\times/(R^0)^\times$ and $\al(\omega^j) \equiv  \al(\omega^j,c) \pmod{ (R^0 \lb u \rb)^\times}$.
\end{lem}
\begin{proof}
By Lemma \ref{lem: L when I principal} \eqref{part:rel},  there is an equality $\al(\omega^j)=\alcv{tx}(\omega^j) t$., which proves \eqref{part:factor}. Applying the homomorphism $c$ yields
\[
t = \al(\omega^j,c) = c(\alcv{tx}(\omega^j))t.
\]
Since $I$ is a free $R^0$-module,  this implies that $c(\alcv{tx}(\omega^j))=1$ and hence that $\alcv{tx}(\omega^j) \in (R^0 \lb u \rb)^\times$ as $R^0 \lb u \rb$ is local.  This proves \eqref{part:unit}. 
\end{proof}

\subsection{Content and $\mu$-invariant}
\label{subsec:axiomatic mu}
These results about the relationships between the modular and cuspidal normalizations have implications about their content and $\mu$-invariants.  We first recall the definitions of these concepts.
\begin{defn}
Let $B$ be a commutative ring $B$. The \emph{content} of a power series $f= \sum a_i(f)u^i \in B \lb u \rb$ is the ideal $\mathrm{content}_B(f) \subset B$ generated by all the coefficients $a_i(f)$; the series $f$ \emph{has unit content} if $\mathrm{content}_B(f)=B$.  Note that, for every $b \in B$, there is an equality
\begin{equation}
\label{eq:content multiplies}
\mathrm{content}_B(bf) = b \cdot \mathrm{content}_B(f),
\end{equation}
and, in particular, that the content of $f$ depends only on the image of $f$ in $B \lb u \rb /B^\times$.

If $B$ is local and $f$ has unit content, then $a_i(f) \in R^\times$ for some $i$, and the \emph{$\lambda$-invariant} $\lambda(f) \in \Z$ is defined as the minimal $i$ such that $a_i(f) \in R^\times$. Note that $f$ is a unit if and only if it has unit content and $\lambda(f)=0$.

Finally,  if $B$ is a DVR with uniformizer $\varpi$,  then the \emph{$\mu$-invariant} $\mu(f) \in \Z$ is defined to the unique integer $n$ such that $\mathrm{content}_R(f)=\varpi^nR$. In this case, $f$ has unit content if and only if $\mu(f)=0$.
\end{defn}

Note that the values of a power series function are in the content ideal. In other words, if $c: B \lb u \rb \to B$ is $B$-algebra homomorphism, then $c(f) \in \mathrm{content}_B(f)$.  Indeed, $c(f) = \sum a_i(f) c(u)^i$ and each $a_i(f)$ is in $\mathrm{content}_B(f)$.

Now let $(\sO,A,R,\E,r_0)$ be a single-Eisenstein Hecke algebra setup and let $(X,\phi,x,\cL)$ be an $L$-symbol setup for it.

\begin{lem}
\label{lem: content when lam(c) generates}
Assume that $I$ is principal, fix $j \in \{0,\dots,p-1\}$, and let $c\colon R\lb u \rb \to~R$ be an $R$-algebra homomorphism.
\begin{enumerate}
\item $L(\omega^j,c) \in \mathrm{content}_R(L(\omega^j))$.
\item There is an equality $\mathrm{content}_R(\al(\omega^j)) =\mathrm{content}_R(\alc(\omega^j))I$ of ideals in $R$.
\item If $L(\omega^j,c)$ is a generator of $I$, then $ \mathrm{content}_R(L(\omega^j))=I$ and $L^0(\omega^j)$ is a unit.
\end{enumerate}
\end{lem}
\begin{proof}
(1) is the general fact that the values of a power series function are in the content ideal.  (2) and (3) follow from Lemmas \ref{lem: L when I principal} and \ref{lem: lam(c) generates}, respectively.
\end{proof}

To discuss $\mu$-invariants, we must work over a DVR, which in our applications involves fixing a weight. We automatize this as follows.

\begin{defn}
\label{defn:fixed weight}
A \emph{fixed-weight setup} is a single-Eisenstein Hecke algebra setup $(\sO,A,R,\E,r_0)$ such that the ideal $I=\ker(\E)$ is principal, together with an $L$-symbol setup $(X,\phi,x,\cL)$ for it, and
\begin{itemize}
\item a generator $t$ of $I$, 
\item an $\sO$-algebra homomorphism $w: A \to \sO$.
\end{itemize}
Given this setup, let $F(X) \in A[X]$ be the characteristic polynomial of $t$, as in Lemma \ref{lem:structure of T} and let $F_w(X) \in \sO[X]$ denote the image of $F(X)$ under $w$.  Let $F_w(X) = \prod_{i=1}^{r_w} F_{w,i}(X)$ be the factorization of $F_w(X)$ in $\sO[X]$ into irreducible polynomials, and let $\sO_{w,i}$ be the normalization of $\sO[X]/(F_{w,i}(X))$, and let $\varpi_{w,i}$ be a uniformizer in $\sO_{w,i}$ and let $X_{w,i}$ be the image of $X$ in $\sO_{w,i}$. Then, by Lemma \ref{lem:structure of T}, the normalization of $R^0_w$ is isomorphic to $\prod_{i=1}^{r_w} \sO_{w,i}$.  For $j \in \{0,\dots,p-2\}$ and $i \in \{1, \dots, {r_w}\}$, let $\al_{w,i}(\omega^j)\in \sO_{w,i}\lb u \rb$ be the image of $\al(\omega^j)$ under the map
\[
R \lb u \rb \to R_w \lb u \rb \to \sO_{w,i}\lb u \rb.
\]
\end{defn}

\begin{rem}
\label{rem:fixed weight}
In our applications, the irreducible factors $F_{w,i}(X)$ correspond to the local-Galois-conjugacy classes of cuspidal eigenforms of weight $w$ that are congruent to the given Eisenstein series, and the integer $r_w$ is the number of such classes. The rings $\sO_{w,i}$ are the valuation ring in their corresponding Hecke fields, and the elements $\al_{w,i}(\omega^j)$ are their $p$-adic $L$-functions.  In this case, we will denote the conjugacy classes by $f_{w,i}$ to evoke the fact that the correspond to modular forms, and similarly denote $\sO_{w,i}$ by $\sO_{f_{w,i}}$, $\varpi_{w,i}$ by $\varpi_{f_{w,i}}$, and $\al_{w,i}(\omega^j)$ by $\al_p^+(f_{w,i},\omega^j)$.
\end{rem}

\begin{lem}
\label{lem:mu_add}
Let $(\sO,A,R,\E,r_0)$ be a single-Eisenstein Hecke algebra setup, let $(X,\phi,x,\cL)$ be an $L$-symbol setup for it and let $w: A \to \sO$ be an $\sO$-algebra homomorphism.  Suppose there is a surjective $R$-algebra homomorphism $c: R \lb u \rb \to R$ such that $\al(\omega^j,c)$ is a generator of $I$.  These data give a fixed-weight setup; let $F_w(X) \in \sO[X]$ and $L_{w,i}(\omega^j) \in \sO_{w,i} \lb u \rb$ be as in Defintion \ref{defn:fixed weight} for this setup. 
\begin{enumerate}
\item
\label{part:mu_add1}
 There is an equality
\[
\mathrm{val}_\varpi(F_w(0)) = \sum_{i=1}^{r_w} \mu(\al_{w,i}(\omega^j)).
\]
\item 
\label{part:mu_add2}
Suppose that there is an integer $M>0$ such that $\mathrm{val}_\varpi(F_w(0))>M r$. Then there exists an $i$ such that $\mu(\al_{w,i}(\omega^j))>M$. 
\end{enumerate}
\end{lem}
\begin{proof}
By Lemma \ref{lem: lam(c) generates}, there is an equality
\[
L(\omega^j) = L(\omega^j,c) L_{tx}^0(\omega^j) = t  L_{tx}^0(\omega^j)
\]
and $ L_{tx}^0(\omega^j) \in R^0\lb u \rb^\times$. Since, for each $i=1,\dots,{r_w}$,  the map
\[
R \to R_w \to \sO_{w,i}
\]
sends $t$ to $X_{w,i}$, it follows that $L_{w,i}(\omega^j)=X_{w,i} U_i$ for a unit $U_i \in \sO_{w,i}\lb u \rb^\times$.  In particular, the $\mu$-invariant of $L_{w,i}(\omega^j)$ is the valuation of $X_{w,i}$:
\[
\mu(L_{w,i}(\omega^j)) = \mathrm{val}_{\varpi_{w,i}}(X_{w,i}).
\]
Since $\sO_{w,i}$ is the normalization of $\sO[X]/(F_{w,i}(X))$,  the valuation of $X_{w,i}$ is given by
\[
\mathrm{val}_{\varpi_{w,i}}(X_{w,i}) = \mathrm{val}_{\varpi}(F_{w,i}(0)).
\]
Combining the last two equalities with the fact that $F_w(0)=\prod_i F_{w,i}(0)$ gives
\[
\mathrm{val}_\varpi(F_w(0)) = \sum_{i=1}^r \mathrm{val}_\varpi(F_{w,i}(0)) = \sum_{i=1}^r \mu(L_{w,i}(\omega^j)),
\]
which proves (1).  Part (2) is clear from (1).
\end{proof}

\begin{exmp}
Suppose that $R^0=A$. Then, by Example \ref{exmp: rank 1}, $F(X)=X+\E(r_0)$, so $F_w(X)=X-w(\E(r_0))$ for every choice of $w$, and $\mu(\al_w(\omega^j))=\mathrm{val}_\varpi(w(\E(r_0)))$.
\end{exmp}

\section{Two-variable $p$-adic $L$-symbols}
\label{sec:pL}
The singular cohomology groups 
$$
H^1(Y_1(Np^r),\Z_p) \mbox{~~and~~} H^1(X_1(Np^r),\Z_p)
$$
can be respectively identified, via Poincar\'e duality with the singular homology groups
$$
H_1(Y_1(Np^r),\{{\rm cusps}\},\Z_p) \mbox{~~and~~} H_1(X_1(Np^r),\Z_p)
$$
as in \cite[Proposition 3.5]{sharifi2011}.    We will make this identification implicitly and write $\{ \alpha, \beta \}_r$ for the cohomology class corresponding to the geodesic connecting $\alpha$ to $\beta$ for $\alpha$ and $\beta$ in $\bP^1(\Q)$.  These groups can also be identified with \'{e}tale (co)homology groups, but one has to be careful about Galois actions, since Poincar\'{e} duality has a one-Tate-twist in it---this is all discussed in \cite[Section 3.5]{sharifi2011}.

Write $H^1(Y_1(Np^r),\Z_p)^{\pm, \ord}$ for the ordinary subspace of this cohomology group with sign $\pm$ (for the action of complex conjugation) and set $\{\alpha,\beta\}_r^{\pm,\ord}$ to be the projection of $\{\alpha,\beta\}_r$ to this subspace.  
Set $H^1_{\Lambda}(Y_1(N)) := \ds \varprojlim_r H^1(Y_1(Np^r),\Z_p)^{\ord}$ and analogously define $H^1_{\Lambda}(X_1(N))$.  We can then write down the Mazur--Kitigawa two-variable $p$-adic $L$-symbol explictly as follows:
$$
\cL_p^\pm := \varprojlim_r \left( \sum_{a \in \left( \Z / p^r \Z \right)^\times} U_p^{-r} \{\infty,a/p^r\}^{\mp,\ord}_r \otimes [a] \right)_r
\in H^1_{\Lambda}(Y_1(N))^\mp \otimes \Z_p \lb \Zpx \rb.
$$
(The sign change here is intentional:\ we want to consider $\cL_p$ as a \emph{functional} on cohomology classes ({\it i.e.}\ a homology class), and the sign change appears because of the one-twist mentioned above in Poincar\'{e} duality.  With this convention, $\cL_p^\pm$ is a functional on cohomology classes of the same sign). By \cite[Proposition 4.3.4]{ohta1999}, we have that $\{\infty,a/p^r\}_r^{\ord}$ is in $H^1(X_1(Np^r),\Z_p)^{\ord}$ and thus $\cL_p^\pm$ actually lives in $H^1_{\Lambda}(X_1(N))^\mp \otimes \Z_p\lb \Zpx \rb$. 

Lastly, let $\bT$ denote the Hida Hecke algebra (to be defined more carefully in the following section).  Then  $H^1_\Lambda(Y_1(N))^\pm$ is a $\bT$-module and for any maximal ideal $\m \subseteq \bT$,  we have that $H^1_\Lambda(Y_1(N))^\pm_\m$ is a direct summand of $H^1_\Lambda(Y_1(N))^\pm$.  We write $\cL_p^\pm(\m)$ for the projection of $\cL_p^\pm$ to $H^1_\Lambda(Y_1(N))^\pm_\m \lb \Zpx \rb$.

\section{The primitive case}
\label{sec:primitive}
In this section, we consider the case Eisenstein families with a \emph{primitive} tame character.  In this case,  congruences modulo $p$ between Eisenstein series and cuspforms arise because $p$ divides an $L$-value, as in \cite{ribet1976}. In the next section, we will consider the case of \emph{trivial} tame character and congruences that occur because $p$ divides an Euler-factor, as in \cite{mazur1978}. This primitive case is the same setting that was considered in \cite{BP2018} and we obtain similar results. The main novelty is that we highlight the role played by the two possible normalizations, modular and cuspidal, which allows us to obtain results and conjectures when the Eisenstein ideal is not assumed to be principal.

\subsection{Setup}
\label{subsec:primitive setup}
  Let $p\ge 5$ and let $N$ be an integer with $p \nmid N\varphi(N)$. Let $\fH$ denote the $p$-adic Hida Hecke algebra of tame level $\Gamma_1(N)$. It is an algebra over $\Z_p \lb (\Z/N\Z)^\times \times \Z_p^\times\rb$ generated by $T_q$ for primes $q \nmid Np$ and $U_\ell$ for $\ell \mid Np$. The Eisenstein ideal $\mathcal{I} \subset \fH$ is the ideal generated by $T_q-(1+\dia{q}q^{-1})$ for $q \nmid Np$ and by $U_\ell-1$ for $\ell \mid Np$.

Let $\m \subset \fH$ be a (Good Eisen) maximal ideal, in the sense of \cite[\S 3.1]{BP2018}, and let $\bT_\m$ be the completion of $\fH$ at $\m$. This determines a character $\theta_\m : (\Z/Np\Z)^\times \to \overline{\Q}_p^\times$, as explained in \emph{loc.\ cit.}, which can be written as $\theta_\m=\omega^{j(\m)}\psi_\m$ with $j(\m) \in \{0,\dots, p-2\}$ and $\psi_\m$ a character of $(\Z/N\Z)^\times$.  The (Good Eisen) condition implies that $\psi_\m$ is primitive.
Let $\sO_\m$ be the valuation ring in the $p$-adic field generated by the values of $\theta_\m$ and let $\varpi \in \sO_\m$ be a uniformizer. Let $\Lambda_\m=\sO_\m \lb T \rb$; it is a regular local flat $\sO_\m$-algebra, and $\bT_\m$ is a finite flat local $\Lambda_\m$-algebra. By \cite[Lemma 3.1]{BP2018}, there is a $\Lambda_\m$-algebra homomorphism $\mathrm{Eis}_\m:\bT_\m \to \Lambda_\m$ with kernel $\mathcal{I}_\m$ that satisfies $\Ann_{\bT_\m}(\I_\m)\cong \Lambda_\m$ and is  generated by the Hecke operator $T_0$ determined by $a_1(T_0f)=a_0(f)$. 

It is straightforward to see that $(\sO,A,R,\mathcal{E},r_0)=(\sO_\m,\Lambda_\m,\bT_\m,\mathrm{Eis}_\m, T_0)$ is a single-Eisenstein Hecke algebra setup in the sense of Definition \ref{def:single eisen}. We note that $\mathcal{E}(r_0) = \mathrm{Eis}_\m(T_0)$ is the  constant term of the Eisenstein family, which in this case is the Kubota-Leopoldt series $L_p(\psi_\m^{-1},\kappa) \in \Lambda_\m$ of \cite[\S 3.8]{BP2018}.

We seek now to form our $L$-symbol setup as in Definition \ref{def:Lsetup}.
Let $C_1(Np^r)$ denote the cusps of $X_1(Np^r)$; following \cite[Section 1.3.2]{FK2024},  cusps lying over the cusp $0$ of $X_0(Np^r)$ will be called \emph{0-cusps}. Let $\cC_r = \ker(\Z_p[C_1(Np^r)(\C)] \stackrel{\Sigma}{\to} \Z_p)$ where $\Sigma$ is the augmentation map.  There is an exact sequence
\begin{equation}
\label{eq:level r boundary sequence}
0 \to H^1(X_1(Np^r),\Z_p) \to H^1(Y_1(Np^r),\Z_p) \xrightarrow{\partial} \cC_r \to 0,
\end{equation}
where $\partial$ is the boundary map at the cusps, which satisfies $\partial(\{\alpha,\beta\}_r)=\alpha-\beta$.
By \cite[Proposition (3.1.2)]{Ohta2003},  the localized inverse limit $\cC_\m = ( \varprojlim_r (\cC_r))_{\m}$ is free of rank one as a $\Lambda$-module, generated by the projection of the class of $0-\infty$ to the $\m$-part.   
Taking inverse limit and localization of the sequences \eqref{eq:level r boundary sequence} then yields an exact sequence
\begin{equation}
\label{eq:Lambda boundary sequence}
0 \to H^1_\Lambda(X_1(N))_{\m} \to H^1_\Lambda(Y_1(N))_{\m} \xrightarrow{\phi} \Lambda_\m \to 0
\end{equation}
where $\phi(\zinf)=1$.  This is the exact sequence of \cite[Theorem (1.5.5) (III)]{Ohta2003} (see also \cite[Section 6.2.5]{FK2024} for the description of $\phi$ as the boundary at 0-cusps).

Since $\zinf \in H^1_\Lambda(Y_1(N))_{\m}^-$,  taking minus-parts of \eqref{eq:Lambda boundary sequence} yields an exact sequence
\begin{equation}
\label{eq:minus part sequence}
0 \to H^1_\Lambda(X_1(N))^-_{\m} \to H^1_\Lambda(Y_1(N))^-_{\m} \xrightarrow{\phi} \Lambda_\m \to 0.
\end{equation}

\begin{prop}
\label{prop:primitive L setup}
Taking $(X,\phi,x,\cL) = (H^1_\Lambda(Y_1(N))^-_{\m},\phi,\zinf,\cL^+_p(\m))$ gives an $L$-symbol setup.
\end{prop}

\begin{proof}
The exact sequence \eqref{eq:minus part sequence} implies that $\ker(\phi)=H^1_{\Lambda}(X_1(N))^-_\m$.  
By \cite[\S 1.7.13 and Proposition 6.3.5]{FK2024}, $H^1_\Lambda(Y_1(N))^-_{\m}$ is a dualizing module for $\Tm$ and $H^1_\Lambda(X_1(N))^-_{\m}$ is a dualizing module for $\Tm^0$. Moreover, these isomorphism are compatible in the sense that the isomorphism $H^1_\Lambda(Y_1(N))^-_{\m} \cong \Hom_{\Lambda_\m}(\Tm,\Lambda_\m)$ sends $H^1_\Lambda(X_1(N))^-_{\m}$ isomorphically to $\Hom_{\Lambda_\m}(\Tm^0,\Lambda_\m)$ (see \cite[Diagram (3.5.3)]{Ohta2003}).
Since $\phi(\zinf)=1$ and $\cL^+_p(\m) \in H^1_{\Lambda}(X_1(N))^-_\m \otimes_{\Z_p} \Z_p \lb\Z_p^\times\rb$, this completes the verification.
\end{proof}

\subsection{Results}
\label{subsec:primitive results}
Applying the results of \S \ref{sec:Axiomatics} yields the following theorem.
\begin{thm}
\label{thm:BP case}
\hfill 
\begin{enumerate}
\item 
\label{part:Tgoren}
If $\Tm$ is Gorenstein, then there is a unique $L^+_p(\m)\in \Tm \lb \Zpx \rb$ such that $\cL^+_p(\m)=L^+_p(\m) \cdot \zinf$.   Moreover, $\mathrm{content}(L^+_p(\m)) \subseteq \mathcal{I}_\m$.
\item 
\label{part:T0goren}
Suppose that $\Tmo$ is Gorenstein,  so that $H^1_\Lambda(X_1(N))^+_\m$ is free of rank 1 over $\Tmo$. Then for every generator $e$ of $H^1_\Lambda(X_1(N))^+_\m$,  there is a unique $L^+_p(\m)^{0}_e\in \bT_\m^0\lb \Zpx \rb$ such that $\cL^+_p(\m)=L^+_p(\m)^{0}_{e} \cdot e$.
\item 
\label{part:principal}
If $\I_\m$ is a principal ideal with generator $t$, then $\Tm$ and $\Tmo$ are Gorenstein and $H^1_\Lambda(X_1(N))^+_\m$ is generated by $t \zinf$. Moreover,  there is an equality $$L^+_p(\m)= L^+_p(\m)^{0}_{t\zinf} t$$ in $\Im \lb \Zpx \rb$.  
\end{enumerate}
\end{thm}

\begin{proof}
Part \eqref{part:Tgoren} is Lemma \ref{lem: L when T goren}.
Part \eqref{part:T0goren} is Lemma \ref{lem: L when T^0 goren}.  
Part \eqref{part:principal} is Lemma \ref{lem: L when I principal}.  \end{proof}
\begin{rem}
Part (3) implies that $t$ divides $L^+_p(\m)$, which is the content of \cite[Theorem 3.14]{BP2018}.  Part (3) is a slight refinement,  in that it identifies the ratio with the cuspidal normalization of the $L$-function.
\end{rem}

As in Section \ref{subsec:axiomatic mu},  in order to discuss $\mu$-invariants, we need to specialize to a particular weight.
Fix an integer $k \ge 2$, let $k:\Lambda_\m \to \sO_\m$ be the weight-$k$ specialization map, and suppose that $U_p-1$ generates $\Im$.  This constitutes a fixed-weight setup,  as in Definition \ref{defn:fixed weight}.  Let $r_k$, $\sO_{f_{k,i}}$,  $\varpi_{f_{k,i}}$, and $L^+_p(f_{k,i})$ be as in Remark \ref{rem:fixed weight} for this setup. 

We write $L^+_p(\m,\omega^j)$ for the projection of $L^+_p(\m)$ to the $\omega^j$-component of $\Tm \lb \Zpx \rb$ and likewise write $L^+_p(\m,\omega^j)^0_e$ for the projection of $L^+_p(\m)^0_e$ to the $\omega^j$-component of $\Tmo \lb \Zpx \rb$.
The following theorem analyzes the $\omega^0$-components of these $p$-adic $L$-functions. Parts (1) and (2) in the theorem are \cite[Theorem 3.21]{BP2018},  with the slight refinement that part (1) identifies the unit $U$ of \emph{loc.~cit.}~with the cuspidal normalization of the $p$-adic $L$-function.

\begin{thm}
\label{thm: U_p-1 gens}
Suppose that $U_p-1$ generates $\mathcal{I}_\m$ and let $e=(U_p-1)\zinf$. Then
\begin{enumerate}
\item 
\label{part:U_p-1 gens 1}
$L^+_p(\m,\omega^0)= L^+_p(\m,\omega^0)^{0}_{e} \cdot (U_p-1)$ and $L^+_p(\m,\omega^0)^{0}_{e} \in (\bT_\m^0\lb u \rb)^\times$.
\item 
\label{part:U_p-1 gens 2}
For every weight $k$ and index $i$,  the $\mu$- and $\lambda$-invariants of $f_{k,i}$ are $$
\mu(L^+_p(f_{k,i},\omega^0)) = \mathrm{val}_{\varpi_{f_{k,i}}}(a_p(f_{k,i})-1) \mbox{ and }
\lambda(L^+_p(f_{k,i},\omega^0)) = 0.
$$
\item 
\label{part:U_p-1 gens 3}
For every weight $k$,  the sum of the $\mu$-invariants is given by
\[
\sum_{i=1}^{r_k} \mu(L^+_p(f_{k,i},\omega^0)) = \mathrm{val}_\varpi(L_p(\psi_\m^{-1},k)).
\]
\item
\label{part:U_p-1 gens 4}
For every integer $M$, there is a weight $k$ and index $i$ such that 
\[\mu(L_p^+(f_{k,i},\omega^0))>M.\]
\end{enumerate}
\end{thm}

\begin{proof}
Let $\mathbf{1}:\Lambda_\m \to \sO_\m$ be evaluation at the trivial character. The argument of \cite[Theorem 3.15]{BP2018} shows that, if $\bT_\m$ is Gorenstein, then 
$$
L^+_p(\m,\mathbf{1}) =L^+_p(\m,\omega^0,\mathbf{1}) = (U_p-1) v
$$
for $v \in (\Tm)^\times$.  Thus part \eqref{part:U_p-1 gens 1} follows from Lemma \ref{lem: lam(c) generates} and immediately implies \eqref{part:U_p-1 gens 2},  whereas part \eqref{part:U_p-1 gens 3}  follows from Lemma \ref{lem:mu_add}\eqref{part:mu_add1}.  Finally let $M>0$ be an integer, and choose an integer $k$ close enough to a zero of $L_p(\psi_\m^{-1},\kappa)$ that 
\[
\mathrm{val}_\varpi(L_p(\psi_\m^{-1},k)))>M \mathrm{rank}_\Lambda(\Tmo).
\]
Since $\mathrm{rank}_\Lambda(\Tmo) \ge r_k$, part \eqref{part:U_p-1 gens 4} follows from Lemma \ref{lem:mu_add}\eqref{part:mu_add2}.
\end{proof}

\subsection{Conjectures} 
Theorem \ref{thm:BP case}(3) implies that, when $\Im$ is principal,  the $p$-adic $L$-function $L_p^+(\m)$ is divisible by a generator of $\Im$. This implies a lower bound on $\mu$-invariants of forms in the Hida family.  Based on the philosophy that $\mu$-invariants should be ``as small as possible'',  it is conjectured in \cite[Conjecture 3.16]{BP2018} that this lower bound is an equality.  To state the conjecture more precisely, we require some notation. For a height-one prime $\p \subset \bT_\m^0$,  let $\sO_\p$ be the normalization of $\Tmo/\p$ (which is a DVR) and let $\varpi_\p$ denote a uniformizer in $\sO_\p$. 
Let $L^+_p(\p)^{0}_e, L^+_p(\p)\in  \sO_\p \lb \Zpx \rb$ denote the images of $L^+_p(\m)^{0}_e$ and $L^+_p(\m)$ (supposing they exist). If $\mathcal{I}_\m$ is generated by $\mathfrak{t}$, let $\mathfrak{t}_\p$ denote the image of $\mathfrak{t}$ in $\sO_\p$. 
\begin{conj}[Bella\"{i}che-Pollack]
\label{conj:BP}
Suppose that $\mathcal{I}_\m$ is generated by $\mathfrak{t}$. Then, for every height-one prime $\p \subset \bT_\m^0$,  there is an equality $\mu(L^+_p(\p,\omega^j))=\mathrm{val}_{\varpi_\p}(\mathfrak{t}_\p)$ for each {\rob even} $j$.
\end{conj}

Since this conjecture uses a generator of $\Im$ in the statement,  there is no immediate generalization when $\Im$ is not principal.  We make the following conjecture,  which has the same spirit that the $p$-adic $L$-function is not divisible by anything more than is implied by Theorem \ref{thm:BP case}. 

\begin{conj}\label{conj:mu conj}
~\begin{enumerate}
\item Suppose that $\bT_\m$ is Gorenstein. For each {\rob even} $j$, there is an equality $$\mathrm{content}_{\bT_\m}(L^+_p(\m,\omega^j)) =\mathcal{I}_\m.$$
\item Suppose that $\Tmo$ is Gorenstein and let $e$ be a generator of $H^1_\Lambda(X_1(N))^-_\m$ as a $\Tmo$-module.
Then, for each {\rob even} $j$,  $L^+_p(\m,\omega^j)^{0}_e$ has unit content.
\end{enumerate}
\end{conj}

This conjecture is a generalization of Conjecture \ref{conj:BP}, in that they are equivalent whenever the assumptions overlap:

\begin{prop}
\label{prop: conjs are equiv}
Suppose that $\mathcal{I}_\m$ is principal and let $j_0 \in \Z/(p-1)\Z$ with $j_0 \equiv 0 \pmod{2}$. Then the following are equivalent:
\begin{enumerate}
\item Conjecture \ref{conj:mu conj} (1) for $j=j_0$.
\item Conjecture \ref{conj:mu conj} (2) for $j=j_0$.
\item Conjecture \ref{conj:BP} for $j=j_0$.
\end{enumerate} 
Moreover, if $U_p-1$ generates $\mathcal{I}_\m$, then these conjectures are all true for $j=0$.
\end{prop}

To prove the proposition, we require some lemmas.
\begin{lem}
\label{lem: equivalences of unit content}
Suppose that $H^1_\Lambda(X_1(N))^-_\m$ is generated by an element $e$ as a $\bT^0_\m$-module. The following are equivalent:
\begin{enumerate}
\item $L^+_p(\m,\omega^j)^{0}_e$ has unit content.
\item for some height-one prime $\p \subset \bT_\m^0$,  the $\mu$-invariant $\mu(L^+_p(\p,\omega^j)^{0}_e)$ vanishes.
\item for all height-one primes $\p \subset \bT_\m^0$,  the $\mu$-invariant $\mu(L^+_p(\p,\omega^j)^{0}_e)$ vanishes.
\end{enumerate}
\end{lem}
\begin{proof}
Clear since $\bT^0_\m$ is a local ring.
\end{proof}

\begin{lem}
\label{lem: mu ineguality}
Suppose that $\mathcal{I}_\m$ is generated by $\mathfrak{t}$ and let $e=\mathfrak{t}\zinf$. Then for every height-one prime $\p \subset \Tmo$, there is an equality
\[
\mu(L^+_p(\p,\omega^j)) = \mu(L^+_p(\p,\omega^j)^{0}_e) + \mathrm{val}_{\varpi_\p}(\mathfrak{t}_\p).
\]
In particular, $\mu(L^+_p(\p,\omega^j)) \ge \mathrm{val}_{\varpi_\p}(\mathfrak{t}_\p)$, with equality if and only if $\mu(L^+_p(\p,\omega^j)^{0}_e)$ is zero.
\end{lem}
\begin{proof}
Clear from Lemma \ref{lem: L when I principal}.
\end{proof}

\begin{proof}[Proof of Proposition \ref{prop: conjs are equiv}]
First note that, since $\Im$ is principal, Lemma \ref{lem: both goren} implies that  both $\bT_\m$ and $\Tmo$ are Gorenstein.  Hence, the hypotheses of the three conjectures are satisfied, and we need to show that their conclusions are equivalent.

Let $\mathfrak{t}$ be a generator of $\mathcal{I}_\m$ and let $e=\mathfrak{t}\zinf$. 
Theorem \ref{thm:BP case}(3) implies $L^+_p(\m,\omega^{j_0})=\mathfrak{t} \cdot L^+_p(\m,\omega^{j_0})^{0}_e$, so by the multiplicative property \eqref{eq:content multiplies} of content, it follows that
\[
\mathrm{content}_{\bT_\m}(L^+_p(\m,\omega^{j_0})) = \mathfrak{t}\cdot\mathrm{content}_{\bT_\m}(L^+_p(\m,\omega^{j_0})^{0}_e) = \mathcal{I}_\m \cdot \mathrm{content}_{\bT_\m}(L^+_p(\m,\omega^{j_0})^{0}_e).
\]
This makes the equivalence of (1) and (2) clear.

Now assume (2),  so that $L^+_p(\m,\omega^{j_0})^{0}_e$ has unit content, and let $\p \subset \Tmo$ be a height-one prime.  Then $\mu(L^+_p(\p,\omega^{j_0})^{0}_e)=0$ by Lemma \ref{lem: equivalences of unit content}. It then follows from Lemma \ref{lem: mu ineguality} that $\mu(L^+_p(\m,\omega^{j_0}))=\mathrm{val}_{\varpi_\p}(\mathfrak{t}_\p)$, proving (3).

Now assume (3), so that $\mu(L^+_p(\m,\omega^{j_0}))=\mathrm{val}_{\varpi_\p}(\mathfrak{t}_\p)$ for all $\p$.  Then $\mu(L^+_p(\p,\omega^{j_0})^{0}_e)$ vanishes by Lemma \ref{lem: mu ineguality}, and so $L^+_p(\m,\omega^{j_0})^{0}_e$ has unit content by Lemma \ref{lem: equivalences of unit content}, proving (2).

For the last claim, if  $U_p-1$ generates $\mathcal{I}_\m$, then Theorem \ref{thm: U_p-1 gens} implies that $L^+_p(\m,\omega^0)^{0}_e$ is a unit and Conjecture \ref{conj:mu conj} (2) is clear for $j=0$.
\end{proof}

\section{Mazur's case}
\label{sec:mazur}

In this section,  we consider the case of tame level $\Gamma_0(N)$ for a prime $N$.  This is, in some sense, the opposite of the previous section:\ whereas in Section \ref{sec:primitive} we considered forms with \emph{primitive} tame character, in this section we consider \emph{trivial} tame character.
The main difference is that the constant term of the relevant Eisenstein series are multiplied by an Euler factor at $N$ and congruences can occur because $p$ divides that Euler factor.  We focus on these kinds of congruences in the most interesting case:\ when $N \equiv 1\pmod{p}$.  This includes the case considered by Mazur in his original article \cite{mazur1978} on the Eisenstein ideal, and so we refer to this setup as the ``Mazur's case''.  We make use of recent advances \cite{PG3,Wake-JEMS,PG5,Deo,lecouturier2016,lecouturier2021}  about the Gorenstein property for the relevant Hecke algebras. This allows us to replace the Gorenstein assumptions in the results of Section \ref{sec:primitive} with some numerical criteria.

Another difference between this setup and the primitive case considered in Section \ref{sec:primitive} is in the way we specify a unique Eisenstein family.  In tame level $\Gamma_0(N)$, the space of ordinary Eisenstein families has rank 2.  A generic basis of the space is given by $U_N$-eigenforms:\ one where $U_N$ acts by 1 and one where $U_N$ acts by $N^{k-1}$.  Since $N \equiv 1 \pmod{p}$, these two families are congruent and cannot be separated by localizing a maximal ideal of the Hecke algebra.
To resolve this issue,  following an idea of Ohta \cite{ohta2014} as in \cite{PG5}, we replace the $U_N$ operator in the Hecke algebra by the Atkin-Lehner involution $w_N$.  The two $w_N$-eigenvector ordinary Eisenstein families are not congruent, so there is a unique Eisenstein family after localizing at a maximal ideal in this new Hecke algebra.  With our assumptions,  only the family with $w_N=-1$ has constant term divisible by $p$, so we focus solely on that family.

The following subsection summarizes the results of Appendix \ref{appendix:Hida} where it is verified that Hida theory works as expected for these modified Hecke algebras.  After that, we verify that there is a unique ordinary Eisenstein family where $w_N$ acts by $-1$ and we compute the constant term of its $q$-expansion.  We explain how to use Ohta's results on $\Lambda$-adic Eichler-Shimura for tame level $\Gamma_1(N)$ to prove the same results for tame level $\Gamma_0(N)$. We then establish various numerical criteria that guarantee that our Hecke algebras are Gorenstein and further ones that guarantee that $U_p-1$ generates the Eisenstein ideal.  With these results in hand, we apply the axiomatic setup of \S \ref{sec:Axiomatics} to deduce our main analytic results.

\subsection{Hida theory with Atkin-Lehner operators} 
Fix an even integer $k_0$ with $0<k_0<p-1$.  For integers $r\ge 0, k\ge 2$ with $k \equiv k_0 \bmod{p-1}$, let $\fH_{k,Np^r}$ denote the subalgebra of $\End_{\Z_p}(M_k(\Gamma_0(N) \cap \Gamma_1(p^r),\Z_p))$ generated by operators $T_q$ for primes $q \nmid Np$, together with $w_N$ and $U_p$. 
Let $\h_{k,Np^r}$ denote the image of $\fH_{k,Np^r}$ in $\End_{\Z_p}(S_k(\Gamma_0(N) \cap \Gamma_1(p^r),\Z_p))$.  Let $\fH_k^\ord$ denote the inverse limit over $r$ of the ordinary part $\fH_{k,Np^r}^\ord$ of $\fH_{k,Np^r}$, and similarly for $\h_k^\ord$. The main result proven in Appendix \ref{appendix:Hida} is that these algebras satisfy the main theorems of Hida theory, just as for the Hecke algebras with $U_N$-operators:
\begin{itemize}
\item (independence of weight) $\fH_k^\ord$ and $\h_k^\ord$ are independent of $k$ (and depend only on $k_0$), and so can be denoted simply by $\fH^\ord$ and $\h^\ord$,
\item (freeness over $\Lambda$) the algebras $\fH^\ord$ and $\h^\ord$ are free $\Lambda$-modules of finite rank,
\item (duality) the pairing $(f,T) \mapsto a_1(Tf)$ is a perfect duality between $\fH^\ord$ and $\Lambda$-adic modular forms (and similarly for $\h^\ord$ and cuspforms),
\item (control) the natural maps $\fH^\ord \to \fH^\ord_{k,Np^r}$ induce isomorphisms $\fH^\ord/\omega_{r,k} \to \fH^\ord_{k,Np^r}$ for a particular $\omega_{r,k} \in \Lambda$ (and similarly for $\h^\ord$).
\end{itemize}

\subsection{Eisenstein series}
\label{sec:eis}
Let $\m \subseteq \fH^{\ord}$ denote the maximal ideal corresponding to the residual representation $1 \oplus \omega^{k_0-1}$ and which contains $w_N+1$.   Write $\Tm$ for the completion of $\fH^{\ord}$ at $\m$ and $\Tmo$ for the completion of $\h^{\ord}$ at $\m$.  Then both $\Tm$ and $\Tmo$ are modules over $\Lambda = \Z_p\lb T\rb$ the Iwasawa algebra.  

Let $E^{\ord}$ denote the family of Eisenstein series whose specialization to a weight $k \equiv k_0 \pmod{p-1}$ is $E^{\ord}_k$, the unique ordinary Eisenstein series of weight $k$ and level $p$ whose constant term is $-(1-p^{k-1}) \frac{B_k}{2k}$.
We wish to promote this family to an eigenfamily of level $Np$ where $w_N$ acts by $-1$.  To this end, note that it is easy to compute the action of $w_N$ on $E^{\ord}_k$ thought of as a form of level $Np$ as this form is old at $N$.  Indeed, for any form $f$ of level $Np$, we have $w_N f = N^{1-k/2} \cdot f \big|_k \left( \begin{smallmatrix}  Na &b  \\ Np & Nd \end{smallmatrix} \right)$ where this matrix has determinant $N$ and thus
\begin{align*}
w_N E^{\ord}_k(z) 
&= N^{1-k/2} \cdot \left(E^{\ord}_k \big|_k \left( \begin{smallmatrix}  Na &b  \\ Np & Nd \end{smallmatrix} \right)\right)(z) \\
&= N^{1-k/2} \cdot \left(E^{\ord}_k \big|_k \left( \begin{smallmatrix} a & b \\ p & Nd \end{smallmatrix} \right)\left( \begin{smallmatrix} N & 0 \\ 0  & 1 \end{smallmatrix} \right)\right) (z) \\
&= N^{1-k/2} \cdot \left(E^{\ord}_k \big|_k \left( \begin{smallmatrix} N & 0 \\ 0  & 1 \end{smallmatrix} \right) \right)(z) \\
&= N^{k/2} \cdot E^{\ord}_k(Nz).
\end{align*}

With this formula in hand, an easy computation shows that 
\[
\E^\pm_k(q) := E^{\ord}_k(q) \pm N^{k/2}E^{\ord}_k(q^N)
\]
is a $w_N$-eigenform with eigenvalue $\pm1$.  In particular, there is a unique Eisenstein series in $M_k(\Gamma_0(Np))$ where $w_N$ acts with sign $\pm1$.  Further note that
the constant term of $\E^\pm_k(q)$ is given by $-(1 \pm N^{k/2})(1-p^{k-1}) \frac{B_k}{2k}$.

 Let $\E^-$ denote the family of Eisenstein series that in weight $k \equiv k_0 \pmod{p-1}$ specializes to $\E^-_k$.  Let $\Eis:\Tm \to \Lambda$ be the homomorphism corresponding to $\E^-$ and let $\Im$ denote the kernel of this map.  As $\E^-$ is the unique Eisenstein family parametrized by $\Tm$,  the ideal $\Ann_{\Tm}(\Im)$ is free of rank one as a $\Lambda$-module.  It is generated by the Hecke operator $T_0$ determined by $a_1(T_0f)=a_0(f)$, so that $\Tmo=\Tm/T_0\Tm$.
 
Taken together, all this implies that $(\sO,A,R,\E, r_0)=(\Z_p,\Lambda,\Tm,\mathrm{Eis}, T_0)$ is a single-Eisenstein Hecke algebra setup. Further, the element $\xi = \Eis(T_0)$ corresponds to the constant term of $\E^-$ which is $\zeta_{p,k_0} \cdot (1-\langle N \rangle^{1/2})$ where $\zeta_{p,k_0} \in \Lambda$ is the $\omega^{k_0}$-branch of the $p$-adic $\zeta$-function and $\langle N \rangle \in \Lambda$ is the element which specializes to $N^k$ in weight $k$.

\subsection{$\Lambda$-adic Eichler-Shimura for tame level $\Gamma_0(N)$}
Consider the $\Lambda$-adic \'etale cohomology groups $H^1_\Lambda(Y_0(N))$ and $H^1_\Lambda(X_0(N))$ with tame level $\Gamma_0(N)$, defined as
\[
H^1_\Lambda(Y_0(N)) = \varprojlim H^1(Y(\Gamma_0(N) \cap \Gamma_1(p^r)),\Z_p)^{\ord, (k_0)},
\]
and similarly for $X_0(N)$, where the superscript ${(k_0)}$ means the $\omega^{k_0}$-eigenspace for the diamond-operator action of $(\Z/p\Z)^\times$.  In this section, we use Ohta's results on the structure of $H^1_\Lambda(Y_1(N))$ and $H^1_\Lambda(X_1(N))$ to deduce analogous results for $H^1_\Lambda(Y_0(N))$ and $H^1_\Lambda(X_0(N))$.
The main input is the following result about the structure of $H^1_\Lambda(Y_1(N))$ as a $\Z_p[\Delta]$-module, where $\Delta=\Gamma_0(N)/\Gamma_1(N) \cong (\Z/N\Z)^\times$ is the group of diamond operators of level $N$.

\begin{lem}
\label{lem:Delta freeness}
The $\Lambda$-adic cohomology group $H^1_\Lambda(Y_1(N))$ is a projective $\Z_p[\Delta]$-module, and the natural maps
\[
H^1_\Lambda(Y_1(N))_\Delta \to H^1_\Lambda(Y_0(N)) \to H^1_\Lambda(Y_1(N))^\Delta
\]
are isomorphisms.
\end{lem}
\begin{proof}
It is enough to prove the result at level $\Gamma_1(Np^r)$ for fixed $r$; taking inverse limits gives the result for $\Lambda$-adic cohomology.  Let $Y_1 = Y_1(Np^r)$ and $Y_0=Y(\Gamma_0(N) \cap \Gamma_1(p^r))$. 

First note that $H^0(Y_1,\Z_p)^\ord$ is zero because $U_p$ acts by $p$ on $H^0(Y_1,\Z_p)$.  Then the following are evident:
\begin{enumerate}
\item $Y_1 \to Y_0$ is an \'etale covering with Galois group $\Delta$,
\item the cohomology $H^i(Y_1,\Z_p)^\ord$ is only supported in degree $i=1$,  and
\item $H^1(Y_1,\Z_p)^\ord$ is $p$-torsion-free.
\end{enumerate}
These three facts are enough to imply that $H^1(Y_1,\Z_p)^\ord$ is a projective $\Z_p[\Delta]$-module, as follows. By the comparison isomorphism between \'etale and Betti cohomology, it is enough to show that the singular cohomology is projective.
Let $C^\bullet(Y_1,\Z_p)$ be the complex of singular cochains on $Y_1$; by (1), it is a bounded complex of flat $\Z_p[\Delta]$-modules.  Hecke operators act on $C^\bullet(Y_1,\Z_p)$,  and, since the ordinary projector is an idempotent, the complex $C^\bullet(Y_1,\Z_p)^\ord$ is still a bounded complex of flat $\Z_p[\Delta]$-modules and its cohomology groups are $H^i(Y_1,\Z_p)^\ord$. 
Hence, there is a perfect complex $C^\bullet$ of $\Z_p[\Delta]$-modules and a quasi-isomorphism
\[
C^\bullet \simeq C^\bullet(Y_1,\Z_p)^\ord
\]
(see \cite[Lemma II.5.1, pg.~47]{mumford2008}).  By (2), this implies that $H^1(Y_1,\Z_p)^\ord$ has finite projective dimension over $\Z_p[\Delta]$, which, together with (3), implies that it is projective (see \cite[Theorems 8.10 and 8.12, pg.\ 152]{brown1994}, for instance, or use the Auslander-Buchsbaum formula). 

The composition of the two maps
\[
H^1(Y_1,\Z_p)^\ord_\Delta \to H^1(Y_0,\Z_p)^\ord \to (H^1(Y_1,\Z_p)^\ord)^\Delta 
\]
is equal to the map induced by multiplication by the norm element of $\Z_p[\Delta]$ on $H^1(Y_1,\Z_p)^\ord$, which is an isomorphism since $H^1(Y_1,\Z_p)^\ord$ is $\Z_p[\Delta]$-projective. The map $H^1(Y_0,\Z_p)^\ord \to (H^1(Y_1,\Z_p)^\ord)^\Delta$ is clearly injective, so this completes the proof.
\end{proof}
This lemma allows us to to prove the following analog of Ohta's $\Lambda$-adic Eichler-Shimura isomorphisms.
\begin{prop}
\label{prop: Gamma_0 ES}
There are split-exact sequences of $\Tm$-modules
\begin{align*}
0 \to \bT_\m^0 \to H^1_\Lambda(Y_0(N))_\m  \to \Hom_\Lambda(\bT_\m,\Lambda) \to 0 \\
0 \to \bT_\m^0 \to H^1_\Lambda(X_0(N))_\m  \to \Hom_\Lambda(\bT_\m^0,\Lambda) \to 0. 
\end{align*}
\end{prop}
\begin{proof}
Let $\mathfrak{H}^\ord_1$ be the $p$-adic Hida Hecke algebra of tame level $\Gamma_1(N)$ and let $\h^\ord_1$ be its cuspidal quotient. Let $m_{\Lambda,1}$ and $S_{\Lambda,1}$ be the spaces of $\Lambda$-adic modular forms and cuspforms, respectively, of tame level $\Gamma_1(N)$, and let $m_\Lambda$ and $S_\Lambda$ be the corresponding spaces of tame level $\Gamma_0(N)$.

Ohta's $\Lambda$-adic Eichler-Shimura isomorphisms \cite{ohta1999,ohta2000} imply that there are exact sequences of $\fH_1^\ord$-modules
\begin{align}
\label{eq:Y es}
0 \to \h_1^\ord \to H^1_\Lambda(Y_1(N)) \to m_{\Lambda,1} \to 0 \\
\label{eq:X es}
0 \to \h_1^\ord \to H^1_\Lambda(X_1(N)) \to S_{\Lambda,1} \to 0.
\end{align}
These sequences are defined using the action of $G_{\Q_p}$ and since $k_0 \ne 1$ (because $k_0$ is even) the actions on the sub and quotient are distinguished,  so the sequences split as $\fH_1^\ord$-modules (see \cite[Section 3.4]{Ohta2003} and \cite[Section 6.3.12]{FK2024}); {\it a fortiori},  they split as $\Z_p[\Delta]$-modules (since the diamond operators are in $\fH_1^\ord$). Lemma \ref{lem:Delta freeness} and the splitting of \eqref{eq:Y es} imply that $\h_{1}^\ord$ and $m_{\Lambda,1}$ are $\Z_p[\Delta]$-projective. Since the dual of a projective $\Z_p[\Delta]$-module is projective, this implies that $\fH_1^\ord$ and $S_{\Lambda,1}$ are projective $\Z_p[\Delta]$-modules. Then the sequence \eqref{eq:X es} implies that $H^1_\Lambda(X_1(N))$ is also $\Z_p[\Delta]$-projective. 
Just as in the proof of Lemma \ref{lem:Delta freeness}, this implies that the natural map
\[
H^1_\Lambda(X_0(N)) \to H^1_\Lambda(X_1(N))^\Delta
\]
is an isomorphism. Hence taking $\Delta$-invariants of the split-exact sequences \eqref{eq:Y es} and \eqref{eq:X es} yields exact sequences
\begin{align}
\label{eq:Y0 es}
0 \to (\h_1^\ord)^\Delta \to H^1_\Lambda(Y_0(N)) \to (m_{\Lambda,1})^\Delta \to 0 \\
\label{eq:X0 es}
0 \to (\h_1^\ord)^\Delta \to H^1_\Lambda(X_0(N)) \to (S_{\Lambda,1})^\Delta \to 0.
\end{align}
Now note that the natural maps $m_\Lambda \to (m_{\Lambda,1})^\Delta$ and $S_\Lambda \to (S_{\Lambda,1})^\Delta$ are isomorphisms because a modular form of level $\Gamma_1(N)$ that is invariant under the diamond operators is also a form of level $\Gamma_0(N)$.  This implies that the dual map
\begin{equation}
\label{eq:Delta-coinvariants on T}
\Hom_\Lambda(S_{\Lambda,1},\Lambda)_\Delta \to \Hom_\Lambda(S_\Lambda,\Lambda), 
\end{equation}
is an isomorphism.  But, by duality (see Theorem \ref{thm:classical hida theory} and Theorem \ref{thm:atkin-lehner control}), there are isomorphisms $\Hom_\Lambda(S_{\Lambda,1},\Lambda) \cong \h_1^\ord$ and $\Hom_\Lambda(S_\Lambda,\Lambda) \cong \h^\ord$. Moreover, since $\h_1^\ord$ is $\Z_p[\Delta]$-free,  the norm map induces an isomorphism
\[
(\h_1^\ord)_\Delta \cong (\h_1^\ord)^\Delta.
\]
Combining this isomorphism with the duality isomorphisms and \eqref{eq:Delta-coinvariants on T} gives a string of isomorphisms
\[
(\h_1^\ord)^\Delta \cong (\h_1^\ord)_\Delta \cong \Hom_\Lambda(S_{\Lambda,1},\Lambda)_\Delta \cong  \Hom_\Lambda(S_\Lambda,\Lambda) \cong \h^\ord.
\]
Hence \eqref{eq:Y0 es} and \eqref{eq:X0 es} are isomorphic to
\begin{align}
\label{eq:Y0 es1}
0 \to \h^\ord \to H^1_\Lambda(Y_0(N)) \to m_{\Lambda} \to 0 \\
\label{eq:X0 es1}
0 \to \h^\ord \to H^1_\Lambda(X_0(N)) \to S_{\Lambda}\to 0.
\end{align}
Localizing at $\m$ then completes the proof.
\end{proof}

Just as in Section \ref{subsec:primitive setup}, there is an exact sequence
\[
0 \to H^1_\Lambda(X_0(N)) \to H^1_\Lambda(Y_0(N)) \xrightarrow{\phi} \Lambda \to 0.
\]
where $\phi(\zinf)=1$. Since $\zinf \in H^1_\Lambda(Y_0(N))^-$, this implies that there is an exact sequence
\[
0 \to H^1_\Lambda(X_0(N))^- \to H^1_\Lambda(Y_0(N))^- \xrightarrow{\phi} \Lambda \to 0.
\]
\begin{prop}
\label{prop:mazur L setup}
Taking $(X,\phi,x,\cL)= ( H^1_\Lambda(Y_0(N))^-, \phi,\zinf,\cL_p^+(\m))$ gives an $L$-symbol setup.
\end{prop}
\begin{proof}
Just as in the proof of Proposition \ref{prop:primitive L setup}, given Proposition \ref{prop: Gamma_0 ES}.
\end{proof}
\subsection{Gorenstein results and generators of the Eisenstein ideal}
\label{sec:goren}

In this section, we give some numerical criteria for when the Eisenstein ideal $\Im$ is principal (and thus $\Tm$ and $\Tmo$ are Gorenstein), when $\Im$ is generated by $U_p-1$, and when $\Tmo$ has rank 1 over $\Lambda$.

Let $\Tmtv{k}$ and $\Tmv{k}$ denote the respective Hecke algebras of $M_k(\Gamma_0(N))^{\ord}$ and $M_k(\Gamma_0(Np))^{\ord}$ and let $\Tmtvo{k}$ and $\Tmvo{k}$ denote the respective Hecke algebras of $S_k(\Gamma_0(N))^{\ord}$ and $S_k(\Gamma_0(Np))^{\ord}$.  Here all of these Hecke algebras are defined to contain $w_N$ rather than $U_N$.

\begin{thm}
\label{thm:no change}
If $k>2$,  there are isomorphisms $\Tmv{k} \cong \Tmtv{k}$ and $\Tmvo{k} \cong \Tmtvo{k}$.  For $k=2$, the same conclusions hold when $p$ is not a $p$-th power modulo $N$.
\end{thm}

\begin{proof}
This result is easy is $k \neq 2$.  Indeed, an eigenform $f$ of weight $k$ which is $p$-new has $a_p(f) = \pm p^{\frac{k-2}{2}}$ and thus is ordinary if and only if $k=2$.  In particular, when $k \neq 2$, none of these ordinary Hecke algebras grow from level $N$ to level $Np$.

The case of $k=2$ is much deeper and follows from the results of the second author and Wang-Erickson,  namely \cite[Theorem 1.4.5 and Proposition A.3.1]{PG5} and relies on the hypothesis that $p$ is not a $p$-th power modulo $N$.
\end{proof}

Returning to the setting of \S \ref{sec:eis}, we have $\Tm$ and $\Tmo$ are the localized Hida Hecke algebras corresponding to the residual representation $1 \oplus \omega^{k_0-1}$ and we now give criteria for when these rings are Gorenstein.  To this end, let  $g$ denote a generator of $\F_N^\times$ and define $\log_N : \F_N^\times \to \F_p$ by $\log_N(g^a) = a \pmod{p}$.  Note that this map is well-defined as $a$ is defined modulo $N-1$ and $p \mid N-1$.
Further, fix $\zeta_p$ a $p$-th root of unity in $\F_N^\times$.

\begin{thm}
\label{thm:principal}
Assume that both of the following two conditions hold:
\begin{enumerate}
\item $p \nmid B_{k_0}$, and
\item $\ds \sum_{i=1}^{p-1} i^{k_0-2} \log_N(1 - \zeta_p^i) \neq 0 \text{~in~} \F_p$.
\end{enumerate}
Then $\Im$ is a principal ideal.  In particular, $\Tm$ and $\Tmo$ are Gorenstein and $\Tmv{k}$ and $\Tmvo{k}$ are Gorenstein for all $k \equiv k_0 \pmod{p-1}$.
\end{thm}

\begin{proof}
The results of \cite{Deo} imply that $\Im$ is principal (see, in particular, \cite[Remark 5.7]{Deo}).  Once we know $\Im$ is principal, then Lemma \ref{lem: both goren} implies that $\Tm$ and $\Tmo$ are Gorenstein and Theorem \ref{thm:atkin-lehner control} implies that $\Tmv{k}$ and $\Tmvo{k}$ are Gorenstein for all $k \equiv k_0 \pmod{p-1}$.
\end{proof}

\begin{rem}
\label{rem:same}
We note that if $k_0 \equiv 2 \pmod{p-1}$, then 
$$
\sum_{i=1}^{p-1} i^{k_0-2} \log_N(1 - \zeta_p^i) =
\sum_{i=1}^{p-1} \log_N(1 - \zeta_p^i) = \log_N\left(\prod_{i=1}^{p=1}1-\zeta_p^i\right) = \log_N(p).
$$
In particular, when $k_0 \equiv 2 \pmod{p-1}$, the hypotheses of Theorem \ref{thm:principal} reduce simply to asking that $p$ is not a $p$-th power modulo $N$
\end{rem}

We now give criteria for when $U_p-1$ generates the Eisenstein ideal.

\begin{thm}
\label{thm:U_p-1}
Assume the hypotheses of Theorem \ref{thm:principal}.  Then $p$ is not a $p$-th power modulo $N$ if and only if  $U_p-1$ generates the Eisenstein ideal.
\end{thm}

\begin{proof}
First note that $U_p-1$ generates the Eisenstein ideal $\Im$ of $\Tm$ if and only if $U_p-1$ generates the Eisenstein ideal $\Imk$ of $\Tmv{k}$ for one (equivalently any) $k \equiv k_0 \pmod{p-1}$.  Thus it suffices to work in $\Tmv{k}$ to establish the above theorem.

To this end, choose $k \equiv k_0 \pmod{p-1}$ with $k>2$.  By Theorem \ref{thm:principal}, $\Imk$ is principal, and hence there is a surjective homomorphism $\Z_p[x] \to \Tmv{k}$ sending $x$ to a generator of $\Imk$. This map induces an isomorphism $\Tmv{k}/p\Tmv{k} \cong \F_p[x]/(x^r)$ for some $r \ge 1$.
On the other hand, \cite[Proposition 4.4.2]{Wake-JEMS} gives a surjective homomorphism $\phi:\Tmv{k} \to \F_p[\epsilon]/(\epsilon^2)$ such that 
\begin{align*}
T_\ell &\mapsto 1 + \ell^{k_0-1} + \epsilon (\ell^{k_0-1} -1)\log_N(\ell), \\
U_p &\mapsto 1 - \epsilon \log_N(p).
\end{align*}
Hence, an element $t \in \Imk$ is a generator if and only if $\phi(t) \ne 0$. It follows that $U_p-1$ generates $\Imk$ if only if $p$ is not a $p$-th power modulo $N$.
\end{proof}

\begin{rem}
We note that when $k_0 \equiv 2 \pmod{p-1}$, all of the hypotheses of Theorem \ref{thm:U_p-1} again reduce to simply assuming that $p$ is not a $p$-th power modulo $N$.
\end{rem}

Lastly, we give a criteria for when $\Tmo$ has rank 1 over $\Lambda$.

\begin{thm}
\label{thm:rank1}
Assume the hypotheses of Theorem \ref{thm:U_p-1} and that
\begin{equation}
\label{eqn:nrk1}
\prod_{i=1}^{N-1} i^{\left(\sum_{j=1}^{i-1} j^{k_0-1}\right)}
\end{equation}
is not a $p$-th power modulo $N$.  Then $\Tmo$ has rank 1 over $\Lambda$.
\end{thm}

\begin{proof}
By Theorem \ref{thm:atkin-lehner control}, to compute the rank of $\Tmo$ over $\Lambda$, it suffices to compute the rank of $\Tmvo{k}$ over $\Z_p$ for any $k \equiv k_0 \pmod{p-1}$.  The theorem then follows from \cite[Corollary B]{Deo}.
\end{proof}

\begin{rem}
In \cite[pg.\ 36]{lecouturier2016}, it is verified that when $k_0 \equiv 2 \pmod{p-1}$ the quantity in \eqref{eqn:nrk1} is a $p$-th power modulo $N$ if and only if Merel's number $\prod_{i=1}^{\frac{N-1}{2}} i^i$ is a $p$-th power modulo $N$.
\end{rem}

\subsection{Results}
By Proposition \ref{prop:mazur L setup},   $(H^1_\Lambda(Y_0(N))^-, \phi,\zinf,\cL_p^+(\m))$ is an $L$-symbol setup for the single-Eisenstein Hecke algebra setup $(\Z_p,\Lambda,\Tm,\E^-,T_0)$. Applying the results of Section \ref{sec:Axiomatics} to this setup, we obtain a theorem which is essentially identical to Theorem \ref{thm:BP case}:

\begin{thm}
\label{thm:Mazur case}
\hfill
\begin{enumerate}
\item 
If $\Tm$ is Gorenstein, then there is a unique $L^+_p(\m)\in \Tm \lb \Zpx \rb$ such that $\cL^+_p(\m)=L^+_p(\m) \cdot \zinf$.  Moreover, $\mathrm{content}_{\bT_\m}(L^+_p(\m)) \subseteq \mathcal{I}_\m$.
\item 
Suppose that $\Tmo$ is Gorenstein,  so that $H^1_\Lambda(X_0(N))^+_\m$ is free of rank 1 over $\Tmo$. Then, for every generator $e$ of $H^1_\Lambda(X_0(N))^+_\m$, there is a unique $L^+_p(\m)^{0}_e\in \bT_\m^0\lb \Zpx \rb$ such that $\cL^+_p(\m)=L^+_p(\m)^{0}_{e} \cdot e$.
\item 
If $\I_\m$ is a principal ideal with generator $t$, then $\Tm$ and $\Tmo$ are Gorenstein and $H^1_\Lambda(X_0(N))^+_\m$ is generated by $t \zinf$. Moreover,  there is an equality $$L^+_p(\m)= L^+_p(\m)^{0}_{t\zinf} t$$ in $\Im \lb \Zpx \rb$.  
\end{enumerate}
\end{thm}

 Note that Theorem \ref{thm:principal} gives criteria for when the hypothesis of part \eqref{part:principal} holds. We now move on to the case where $U_p-1$ generates the Eisenstein ideal.  We use notation analogous to that used in Section \ref{subsec:primitive results}. In particular, for a fixed weight $k$, the forms $f_{k,1},\dots, f_{k,r_k}$ are a complete list of Galois-conjugacy-classes of eigenforms for $\bT_{\m,k}$ and $\varpi_{f_{k,i}}$ denotes a uniformizer in the $p$-adic Hecke field of $f_{k,i}$.

\begin{thm}
\label{thm:Mazur U_p-1 gens}
Let $e=(U_p-1)\zinf$. Assume that all of the following three conditions hold:
\begin{enumerate}[label=(\alph*)]
\item \label{mazur2}
$p \nmid B_{k_0}$,
\item \label{mazur4}
$\ds \sum_{i=1}^{p-1} i^{k_0-2} \log_N(1 - \zeta_p^i) \neq 0$ in $\F_p$, and
\item \label{mazur3}
$p$ is not a $p$-th power modulo $N$.
\end{enumerate}
Then the following four statements are all true:
\begin{enumerate}
\item 
\label{part:U_p-1 gens 1}
$L^+_p(\m,\omega^0)= L^+_p(\m,\omega^0)^{0}_{e} \cdot (U_p-1)$ and $L^+_p(\m,\omega^0)^{0}_{e} \in (\bT_\m^0\lb u \rb)^\times$.
\item 
\label{part:U_p-1 gens 2}
For every weight $k$ and index $i$,  the $\mu$- and $\lambda$-invariants of $f_{k,i}$ are
$$
\mu(L^+_p(f_{k,i},\omega^0)) = \mathrm{val}_{\varpi_{f_{k,i}}}(a_p(f_{k,i})-1) \mbox{ and }
\lambda(L^+_p(f_{k,i},\omega^0)) = 0.
$$
\item 
\label{part:U_p-1 gens 3}
For every weight $k$,  the sum of the $\mu$-invariants is given by
\[
\sum_{i=1}^{r_k} \mu(L^+_p(f_{k,i},\omega^0)) = \mathrm{val}_p(N-1) + \mathrm{val}_p(k).
\]
\item
\label{part:U_p-1 gens last}
For every integer $M$ and every weight $k$ such that $\mathrm{val}_p(k)>M\mathrm{rank}_\Lambda(\Tmo)$, there exists an index $i$ such that
\[
 \mu(L^+_p(f_{k,i},\omega^0))>M.
\]
\end{enumerate}
If,  in addition,
$$
\prod_{i=1}^{N-1} i^{\left(\sum_{j=1}^{i-1} j^{k_0-1}\right)} \mbox{ is not a } p\mbox{-th power},
$$
then $\rank_\Lambda \Tmo =1$ and 
$$
\mu(L^+_p(f_{k},\omega^0)) = \mathrm{val}_p(N-1) + \mathrm{val}_p(k),
$$
where $f_k$ is the unique form of weight $k$ in the Hida family corresponding to the isomorphism $\Tmo\cong \Lambda$.
\end{thm}

\begin{proof}
By Theorem \ref{thm:U_p-1}, the assumptions (a)-(c) imply that $U_p-1$ generates $\Im$. 
The proof of parts \eqref{part:U_p-1 gens 1}-\eqref{part:U_p-1 gens last} works verbatim as in the proof of Theorem \ref{thm: U_p-1 gens}, except that the role of the Kubota--Leopold series $L_p(\psi^{-1}_\m,\kappa)$ is played by $\E^-(T_0)$, which equals $\zeta_{p,k_0}(\kappa)(1-N^{\kappa/2})$. In particular,  for every weight $k$,
\begin{align*}
\mathrm{val}_\varpi(\E^-(T_0)|_{\kappa=k})&=\mathrm{val}_\varpi(\zeta_{p,k_0}(k)) +\mathrm{val}_\varpi(1-N^{k/2}) \\
&=\mathrm{val}_\varpi(N-1) + \mathrm{val}_\varpi(k),
\end{align*}
because $\mathrm{val}_\varpi(\zeta_{p,k_0}(k))=0$ by \ref{mazur2}. The last claim follows from Theorem \ref{thm:rank1} and~\eqref{part:U_p-1 gens 3}.
\end{proof}

\begin{rem}
We note that the relative advantage of Theorem \ref{thm:Mazur U_p-1 gens} over Theorem \ref{thm: U_p-1 gens} is that all of the hypotheses are simple numerical criteria that can be verified to be true in any given case.
\end{rem}

\subsection{Conjectures and evidence}

The below is a re-statement of Conjecture \ref{conj:mu conj} but now in the setting of this section.

\begin{conj}\label{conj:Mazur mu conj}
~\begin{enumerate}
\item Suppose that $\bT_\m$ is Gorenstein. For each {\rob even} $j$, there is an equality $$\mathrm{content}_{\bT_\m}(L^+_p(\m,\omega^j)) =\mathcal{I}_\m.$$
\item  Suppose that $\Tmo$ is Gorenstein with $H^1_\Lambda(X_1(N))^+_\m$ generated by an element $e$ as a $\Tmo$-module.
Then $L^+_p(\m,\omega^j)^{0}_e$ has unit content {\rob for $j$ even}.
\end{enumerate}
\end{conj}

If $\Im$ is principal then both $\Tm$ and $\Tmo$ are Gorenstein and the hypotheses of both parts of the above conjecture are satisfied.  When $U_p-1$ generates $\Im$, then the $j=0$ case of this conjecture follows from Theorem \ref{thm:Mazur U_p-1 gens}.  

We ran numerical tests using Sage \cite{sagemath} of this conjecture analogous to what was done in \cite[Section 3.9]{BP2018}.  Namely, for all primes $N<80$, we ran through all primes $p\geq 5$ such that $N \equiv 1 \pmod{p}$.  For each such pair $(N,p)$, we considered $2 \leq k_0 \leq p-3$ and when the corresponding $\Tmo$ had rank 1, we verified the above conjecture for all values of $j$.
The assumption of the rank 1 of $\Tmo$ is needed because the method of computation involved iterating $U_p$ on overconvergent modular symbols which converges when there is a unique form in the Hida family in each weight $k$.  This assumption on the rank of $\Tmo$ also implies that both $\Tmo$ and $\Tm$ are Gorenstein and so $\Im$ is principal in this case.
In all, we verified the conjecture (for all allowable values of $j$) for 35 eigenforms.

\appendix

\section{Hida theory with tame Atkin-Lehner involutions}
\label{appendix:Hida}
We consider a variant of the Hida Hecke algebra with operators $w_\ell$ at primes $\ell$ dividing the tame level rather than $U_\ell$. Let $N=p\ell_1\dots\ell_r$ be squarefree, and assume $p>3$.

\subsection{Review of Hida theory} In this section, we review Hida theory. All of these results are well-known, see, for example, \cite{hida1986}, \cite{hida1986a}, \cite[Chapter 7]{hida1993}, \cite{wiles1988}, \cite[Section 1.5]{FK2024}.

\subsubsection{Classical modular forms} For $r \ge 0$ let $\Gamma_r$ denote the congruence subgroup $\Gamma_0(N) \cap \Gamma_1(p^r)$. For a ring $R$, let $M_k(\Gamma_r,R)$ and $S_k(\Gamma_r,R)$ denote the spaces of modular forms and cusp forms, respectively, of weight $k$ and level $\Gamma_r$,  with coefficients in $R$. There is an injective $q$-expansion map
\[
\mathrm{expand}_q: M_k(\Gamma_r,R) \to R\lb q \rb 
\]
which we denote by $f \mapsto \sum_n a_n(f)q^n$. For $f \in S_k(\Gamma_r,R)$ we have $a_0(f)=0$. We define
\[
m_k(\Gamma_r,R) = \{f \in M_k(\Gamma_r,Q(R)) \ : \ a_n(f) \in R \text{ for all } n>0\}
\]
where $Q(R)$ is the localization $S^{-1}R$ where $S$ is the set of non-zero-divisors of $R$. 

There are Hecke operators $T_n$ for $(n,N)=1$ and $U_\ell$ for $\ell \mid N$ as well as the Atkin-Lehner involutions $w_\ell$ for $\ell \mid N$ and diamond operators $\dia{n}$ for $(n,N)=1$. These all act on $m_k(\Gamma_r,R)$ and preserve the submodules $M_k(\Gamma_r,R)$ and $S_k(\Gamma_r,R)$. We let
\[
\fH'_{r,k} \subset \End_{\Z_p}(m_k(\Gamma_r,\Z_p)), \ \h'_{r,k} \subset \End_{\Z_p}(S_k(\Gamma_r,\Z_p))
\]
be the $\Z_p$-subalgebras generated by the diamond operators and $T_n$ for $(n,N)=1$ as well as $U_\ell$ for $\ell | N$. These are commutative algebras, and the subject of Hida theory.

We let
\[
\fH_{r,k} \subset \End_{\Z_p}(m_k(\Gamma_r,\Z_p)), \ \h_{r,k} \subset \End_{\Z_p}(S_k(\Gamma_r,\Z_p))
\]
be the $\Z_p$-subalgebras generated by the diamond operators and $T_n$ for $(n,N)=1$ as well as $U_p$ and $w_\ell$ for $\ell | \frac{N}{p}	$. These are commutative algebras, and are the main focus in this paper.

\begin{lem}
\label{lem:classical duality}
Let $M$ be $m_k(\Gamma_r,\Z_p)$ or $S_k(\Gamma_r,\Z_p)$, and let $H$ be $\fH'_{r,k}$ or $\h'_{r,k}$, respectively. Then $M$ and $H$ are free $\Z_p$-modules of finite rank and the pairing
\[
M \times H \to \Z_p
\]
given by $(f,T) \mapsto a_1(Tf)$ is perfect.
\end{lem}

\subsubsection{Cohomology} Let $Y(\Gamma_r)$ be the modular curve with level $\Gamma_r$ and let $X(\Gamma_r)$ be its compactification. For $k \ge 2$, let $\sF_k$ denote the twisted constant $p$-adic \'etale sheaf on $Y(\Gamma_r)$ associated to the representation $\mathrm{Sym}^{k-2} \mathrm{Std}$ of $\mathrm{GL}_2$, and denote by the same letter its pushforward to $X(\Gamma_r)$. 
Consider the cohomology groups $H^1(r,k), H^1_c(r,k), H^1_P(r,k)$ defined as follows
\[
H^1(r,k):=H^1(Y(\Gamma_r),\sF_k), \  H^1_c(r,k):=H^1_c(Y(\Gamma_r),\sF_k), \ H^1_P(r,k):=H^1(X(\Gamma_r),\sF_k).
\]

By Eichler-Shimura theory, the algebras $\fH_{r,k}'$ and $\fH_{r,k}$ act on $H^1_\dagger(r,k)$ for any $\dagger \in \{\emptyset,c,P\}$, faithfully if $\dagger \in  \{\emptyset,c\}$, and factoring exactly through $\h_{r,k}'$ and $\h_{r,k}$, respectively, if $\dagger=P$.
\begin{lem}
\label{lem:compat of w_ell}
For any $k\ge 2$, any $r'\ge r \ge 0$, and any $\dagger \in \{\emptyset,c,P\}$, the trace maps
\[
H^1_\dagger(r',k) \to H^1_\dagger(r,k)
\]
commute with the actions of $T_\ell$ and $\dia{\ell}$ for any $\ell \nmid N$, $U_\ell$ for any $\ell \mid N$, and $w_\ell$ for any $\ell \mid\mid N$.
\end{lem}
\begin{proof}
For $T_\ell$ and $U_\ell$, this is proven in \cite[Lemma 7.4.1]{ohta1993}, and essentially the same proof works for $w_\ell$. Indeed, just as in that proof, it is enough to that $w_\ell$ commutes with the natural map
\[
H^1_\dagger(r,k) \to H^1_\dagger(r',k)
\]
for $\dagger\in \{\emptyset,c\}$. This commutivity is clear because on both spaces $w_\ell$ is given by the same double coset operator. Explicitly, at any level $M$ with $\ell||M$, $w_\ell$ is the double coset operator associated to any matrix $W_{\ell,M}$ of the form
\[
W_{\ell,M}=\ttmat{\ell x}{y}{Mz}{\ell w}
\]
such that $\det(W_{\ell,M})=\ell$ (the operator is independent of the choice of $W_{\ell,M}$). We see that any choice of $W_{\ell,Np^{r'}}$ is also a valid choice for $W_{\ell,Np^r}$.
\end{proof}

\subsubsection{Hida theory} For a $\fH_{r,k}$-module or $\fH_{r,k}'$-module $M$, let $M^\ord$ denote the largest direct summand on which $U_p$ acts invertibly. For $k$ fixed, let
\[
\fH'^\ord_k = \varprojlim_{r \ge 0} \fH'^\ord_{r,k}, \ \h'^\ord_k = \varprojlim_{r \ge 0} \h'^\ord_{r,k}
\]
where the transition maps send $T_q$ and $U_\ell$ to the operator with the same name (this is well-defined by Lemma \ref{lem:compat of w_ell}). These are algebras over the Iwasawa algebra $\Lambda = \Z_p \lb \Z_p^\times \rb$, via the diamond operator action.  For $a=1, \dots, p-1$, and any $\Lambda$-module, $M$, let $M^{(a)}$ denote the direct summand where the torsion subgroup of $\Z_p^\times$ acts by the $a$-th power of the Teichemuller character. Identify $\Lambda^{(a)}$ with $\Z_p \lb 1+p\Z_p \rb$, and, for any $k \equiv a \pmod{p}$, let $\omega_{r,k} = [1+p]^{p^r}-(1+p)^{(k-2)p^r} \in \Lambda^{(a)}$.

Let
\[
m_{k,\Lambda}^\ord = \varprojlim_{r \ge 0} m_k(\Gamma_r,\Z_p)^\ord, \ S_{k,\Lambda}^\ord = \varprojlim_{r \ge 0} S_k(\Gamma_r,\Z_p)^\ord
\]
where the transition maps are the trace maps. These are faithful modules for $\fH'^\ord_k$ and $\h'^\ord_k$, respectively. The maps
\[
m_k(\Gamma_r,\Z_p)^\ord \to \Q_p[(\Z/p^r\Z)^\times]\lb q \rb
\]
given by
\[
f \mapsto \sum_{a \in (\Z/p^r\Z)^\times} (\mathrm{expand}_q(\dia{a}^{-1}U_p^rw_{Np^r}(f)) [a]
\]
induce injective $\Lambda$-module homomorphisms
\[
m_{k,\Lambda}^\ord \to Q(\Lambda) + q \Lambda\lb q \rb, \ S_{k,\Lambda}^\ord \to q \Lambda\lb q \rb,
\]
and we identify these modules with there images in $Q(\Lambda)\lb q \rb$. The following is known as Hida's control theorem.

\begin{thm}
\label{thm:classical hida theory}
The $\Lambda$-modules $m_{k,\Lambda}^\ord,S_{k,\Lambda}^\ord,\fH'^\ord_k$ and $\h'^\ord_k$ are independent of the integer $k\ge2$. We may and do drop the subscript $k$ from the notation, and denote these objects simply by $m_{\Lambda}^\ord,S_{\Lambda}^\ord,\fH'^\ord$ and $\h'^\ord$. 

Let $M$ be $m_{\Lambda}^\ord$ or $S_{\Lambda}^\ord$ and let $H$ be $\fH'^\ord$ or $\h'^\ord$, respectively. Let $M_{r,k}$ and $H_{r,k}$ be the fixed weight and level versions. We have:
\begin{enumerate}
\item $M$ and $H$ are free $\Lambda$-modules of finite rank.
\item The pairing 
\[
M \times H \to \Lambda,
\]
given by $(f,T) \mapsto a_1(Tf)$, is perfect.
\item For any $k\ge 2$ and $r\ge0$, the natural maps
\[
M/\omega_{r,k}M \to M_{r,k}, \ H/\omega_{r,k}H \to H_{r,k}
\]
are isomorphisms.
\end{enumerate}
\end{thm}

\subsection{Hida theory with Atkin-Lehner operators} In this section, we prove the analog of the control theorem for the algebras $\fH_{r,k}^\ord$ and $\h_{r,k}^\ord$.

For an element $\epsilon=(\epsilon_1,\dots,\epsilon_r) \in \{\pm 1\}^r$ and an $\fH_{r,k}$-module $M$, let $M^\epsilon$ denote the summand of $M$ on which $w_{\ell_i}$ acts by $\epsilon_i$ for $i=1,\dots, r$.

\subsubsection{Atkin-Lehner theory, and reducing from level $N$ to level $N/p$} We will frequently make use of the following result, which is a variant of Atkin-Lehner theory (c.f.~\cite[Theorem 1]{AL1970}) that allows for more general coefficient rings. It was first proven by Mazur \cite[Lemma II.5.9, pg.~83]{mazur1978} in the case $M>5$ is prime, and in the general case by Ohta \cite[Corollary (2.1.4)]{ohta2014}.
\begin{lem}
\label{lem:ohta Atkin-Lehner} Let $M$ be an integer and let $f \in M_k(\Gamma_0(M),R)$ with $R$ a $\Z[1/2M]$-algebra, and suppose that $f$ is an eigenform for all $w_\ell$ with $\ell|M$ and that $a_n(f)=0$ for all $(n,M)=1$. Then $f$ is constant.
\end{lem}
Note that the lemma requires that the level be invertible in the coefficient ring. We want to apply the lemma to rings $R$ where $p \notin R^\times$. The following result, which essentially says that ordinary forms are old-at-$p$ when the weight is at least 3, is useful to remove $p$ from the level. It was proven by Gouv\^{e}a \cite[Lemma 3]{gouvea1992} for cusp forms, and the same proof works for modular forms (c.f.~
\cite[Proposition 1.3.2]{ohta2005}).
\begin{lem}
\label{lem:gouvea}
Let $M$ be any integer that is prime to $p$ and let $k>2$. Then the map
\[
M_k(\Gamma_1(M),\Z_p)^{\ord} \to M_k(\Gamma_1(M) \cap \Gamma_0(p),\Z_p)^\ord
\]
is an isomorphism.
\end{lem}

\subsubsection{Duality} The key to the proof of the control theorem will be the following duality result, which is the analog of Lemma \ref{lem:classical duality}.
\begin{prop}
\label{prop:Atkin-Lehner duality}
For any $r \ge 0$ and $k\ge 2$, consider the pairings
\[
(-,-)_{r,k}:m_k(\Gamma_r,\Z_p)^{\ord} \times \fH_{r,k}^{\ord} \to \Z_p, \ (-,-)^0_{r,k}:S_k(\Gamma_r,\Z_p)^{\ord} \times \h_{r,k}^{\ord} \to \Z_p
\]
given by $(f,T) \mapsto a_1(Tf)$.  
\begin{enumerate}
\item For any $r,k$, the resulting maps
\[
\fH_{r,k}^{\ord} \to \Hom_{\Z_p}(m_k(\Gamma_r,\Z_p)^{\ord},\Z_p), \ \h_{r,k}^{\ord} \to \Hom_{\Z_p}(S_k(\Gamma_r,\Z_p)^{\ord},\Z_p)
\]
are injective.
\item For $r=0$ and any $k>2$, the pairings are perfect.
\end{enumerate}
\end{prop}
For the proof of the proposition, we need a lemma comparing $m_k(\Gamma_0,\Z_p)^{\ord,\epsilon}$ to $M_k(\Gamma_0,\Z_p)^{\ord,\epsilon}$.
\begin{lem}
\label{lem:mk def}
For any $\epsilon \in \{\pm 1\}^r$, there is an exact sequence
\[
0 \to M_k(\Gamma_0,\Z_p)^{\ord,\epsilon} \to m_k(\Gamma_0,\Z_p)^{\ord,\epsilon} \xrightarrow{a_0} \left(\frac{1}{p^{m_{k,\epsilon}}}\Z_p\right)/\Z_p \to 0
\]
for some integer $m_{k,\epsilon} \ge 0$. Moreover, if we define $m'$ as the maximal integer $m$ such that there exists $h \in M_k(\Gamma_0,\Z_p)^{\ord,\epsilon}$ with $h \equiv 1 \pmod{p^m}$ and $a_0(h)=1$, then $m_{k,\epsilon}=m'$.
\end{lem}
\begin{proof}
We have an exact sequence
\[
0 \to M_k(\Gamma_0,\Z_p)^{\ord,\epsilon} \to m_k(\Gamma_0,\Z_p)^{\ord,\epsilon} \xrightarrow{a_0} \Q_p/\Z_p.
\]
Since $m_k(\Gamma_0,\Z_p)^{\ord,\epsilon}$ is finitely generated, the image of the $a_0$ map is a finite cyclic $p$-group. We define $m_{k,\epsilon} \ge 0$ so that the order of this group is $p^{m_{k,\epsilon}}$, and we get the desired exact sequence. 

If $h' \in m_k(\Gamma_0,\Z_p)^{\ord,\epsilon}$ satisfies $a_0(h')=\frac{1}{p^{m_{k,\epsilon}}}$, then $h:=p^{m_{k,\epsilon}}h' \in M_k(\Gamma_0,\Z_p)^{\ord,\epsilon}$ satisfies $h \equiv 1 \pmod{p^{m_{k,\epsilon}}}$ and $a_0(h)=1$, so $m_{k,\epsilon} \le m'$. Conversely, if $h \equiv 1 \pmod{p^{m'}}$ and $a_0(h)=1$, then $h':=p^{-m'}h \in m_k(\Gamma_0,\Z_p)^{\ord,\epsilon}$ and $a_0(h')=p^{-m'}$ so $m_{k,\epsilon} \ge m'$.
\end{proof}
\begin{rem}
It is known that $m_{k,\epsilon}=0$ if $(p-1) \nmid k$, but we shall not use this.
\end{rem}
\begin{proof}[Proof of Proposition \ref{prop:Atkin-Lehner duality}]
We give the proof only for the modular case, the cuspidal case being similar but easier. We can decompose $m_k(\Gamma_0,\Z_p)^{\ord}$ and $\fH_{0,k}^{\ord}$ into eigenspaces for the Aktin-Lehner operators. It is enough to show that, for any choice of $\epsilon$, the induced pairing
\[
(-,-)_{r,k}:m_k(\Gamma_r,\Z_p)^{\ord,\epsilon} \times \fH_{r,k}^{\ord,\epsilon} \to \Z_p,
\]
has the desired properties.

We first show that the pairing is perfect with $\Q_p$-coefficients, which will imply (1). Suppose that $f \in M_k(\Gamma_r,\Q_p)^{\ord,\epsilon}$ satisfies $(f,T)_{r,k}=0$ for all $T \in \fH_{r,k}^{\ord,\epsilon}[1/p]$. This implies that $a_n(f)=0$ for all $(n,N/p)=1$. By Lemma \ref{lem:ohta Atkin-Lehner}, this implies that $f$ is constant, and hence $0$. Conversely, suppose that $T \in \fH_{r,k}^{\ord,\epsilon}[1/p]$ satisfies $(f,T)_{r,k}=0$ for all $f \in M_k(\Gamma_r,\Q_p)^{\ord,\epsilon}$. Then for any $g \in M_k(\Gamma_r,\Q_p)^{\ord,\epsilon}$, we have $a_n(Tg)=(T_ng,T)_{r,k}=0$ for all $(n,N/p)=1$. The same lemma then implies that $Tg=0$, which implies that $T=0$ since $\fH_{r,k}^{\ord,\epsilon}[1/p]$ acts faithfully on $M_k(\Gamma_r,\Q_p)^{\ord,\epsilon}$.

Now we consider (2). By the perfectness of the $\Q_p$-pairing, we have that
\[
m_k(\Gamma_0,\Z_p)^{\ord,\epsilon} \to \Hom_{\Z_p}(\fH_{0,k}^{\ord,\epsilon}, \Z_p)
\]
is injective. We have to show it is surjective. Let $\phi : \fH_{0,k}^{\ord,\epsilon} \to \Z_p$. We see that there is $f' \in M_k(\Gamma_0,\Q_p)^{\ord,\epsilon}$ such that $\phi(T)=(T,f')_k$ for all $T \in \fH_{0,k}^{\ord,\epsilon}$. This implies that $a_n(f')=\phi(T_nf') \in \Z_p$ for all $(n,N/p)=1$, where $T_n:=U_{p^r}T_{n/p^r}$ if $p^r||n$. If $f' \in m_k(\Gamma_0,\Z_p)^{\ord,\epsilon}$, we are done, so assume for a contradiction that it is not. Then there is a minimal $r \ge 1$ such that $f=p^rf' \in m_k(\Gamma_0,\Z_p)^{\ord,\epsilon}$. By the minimality of $r$, we must have $a_n(f) \in \Z_p^\times$ for some $n>0$ (clearly this $n$ must have $(n,N/p)>1$).

Assume that $f \in M_k(\Gamma_0,\Z_p)^{\ord,\epsilon}$. By Lemma \ref{lem:gouvea}, we may consider $f \in M_k(\Gamma_0(N/p),\Z_p)^{\ord,\epsilon}$. Consider the image $\bar{f}$ in $M_k(\Gamma_0(N/p),\F_p)^{\ord,\epsilon}$. Since $a_n(f') \in \Z_p$ for all $(n,N)=1$ and since $r \ge 1$, we have $a_n(\bar{f})=0$ for all $(n,N)=1$. As $N/p \in \F_p^\times$, Lemma \ref{lem:ohta Atkin-Lehner} applies and we get that $\bar{f}$ is constant. This implies that $a_n(f) \equiv 0 \pmod{p}$ for all $n>0$, a contradiction.

Finally, assume that $f \not \in M_k(\Gamma_0,\Z_p)^{\ord,\epsilon}$, so $s:=-\mathrm{val}_p(a_0(f))>0$. Then $g:=a_0(f)^{-1}f \in M_k(\Gamma_0,\Z_p)^{\ord,\epsilon}$ satisfies $a_0(g)=1$. By Lemma \ref{lem:gouvea}, we may consider $g \in M_k(\Gamma_0(N/p),\Z_p)^{\ord,\epsilon}$. Since $a_n(f') \in \Z_p$ for all $(n,N)=1$, we have $a_n(g)=0 \pmod{p^{r+s}}$ for all $(n,N)=1$. As $N/p \in (\Z/p^{r+s}\Z)^\times$, Lemma \ref{lem:ohta Atkin-Lehner} implies that $g \pmod{p^{r+s}}$ is the constant $a_0(g)=1$. By Lemma \ref{lem:mk def}, and since $r>0$, this implies that $s<m_{k,\epsilon}$.

Now fix $h \in M_k(\Gamma_0,\Z_p)^{\ord,\epsilon}$ such that $h \equiv 1 \pmod{p^{m_{k,\epsilon}}}$ and such that $a_0(h)=1$. Since $\mathrm{val}_p(a_0(f))=-s$ and $s<m_{k,\epsilon}$, we see that $a_n(a_0(f)h) \in p\Z_p$ for all $n>0$. Then letting $f''=f-a_0(f)h$, we see that $f'' \in M_k(\Gamma_0(N/p),\Z_p)^{\ord,\epsilon}$. Moreover, we see that $a_n(f'')=p^ra_n(f')-a_n(a_0(f)h) \equiv 0 \pmod{p}$, for all $(n,N/p)=1$. By Lemma \ref{lem:ohta Atkin-Lehner} this implies that $f''\pmod{p}$ is constant. In particular, for all $n>0$ we have
\[
0 \equiv a_n(f'')\equiv a_n(f) \pmod{p},
\]
a contradiction.
\end{proof}

\subsubsection{Control theorem} For $k$ fixed, let
\[
\fH^\ord_k = \varprojlim_{r \ge 0} \fH^\ord_{r,k}, \ \h^\ord_k = \varprojlim_{r \ge 0} \h^\ord_{r,k}
\]
where the transition maps send $T_q$ and $w_\ell$ to the operator with the same name (this is well-defined by Lemma \ref{lem:compat of w_ell}). These are algebras over the Iwasawa algebra $\Lambda = \Z_p \lb \Z_p^\times \rb$, via the diamond operator action. We can now prove the main theorem, which is a control theorem for the algebras $\fH^\ord_k$ and $\h^\ord_k$.
\begin{thm}
\label{thm:atkin-lehner control}
The algebras $\fH^{\ord}_k$ and $\h^\ord_k$ for various $k$ are canonically identified with each other. We can and do drop the subscript $k$ from the notation and simply refer to these algebras as $\fH^{\ord}$ and $\h^\ord$. Moreover:
\begin{enumerate}
\item There are canonical isomorphisms of $\Lambda$-modules $\fH^{\ord} \to \fH'^{\ord}$ and $\h^{\ord} \to \h'^{\ord}$. In particular, $\fH^{\ord}$ and $\h^\ord$ are free $\Lambda$-modules of finite rank.
\item The pairings 
\[
m_\Lambda^\ord \times \fH^{\ord} \to \Lambda, \ S_\Lambda^\ord \times \h^{\ord} \to \Lambda,
\]
given by $(f,T) \mapsto a_1(Tf)$, are perfect.
\item The natural maps 
\[
\fH^{\ord}/\omega_{r,k}\fH^{\ord} \to \fH^{\ord}_{r,k}, \ \h^{\ord}/\omega_{r,k} \h^{\ord} \to \h^{\ord}_{r,k}
\]
are isomorphisms for all $r\ge0$ and $k\ge2$.
\item The pairings $(-,-)_{r,k}$ and $(-,-)_{r,k}^0$ of Proposition \ref{prop:Atkin-Lehner duality} are perfect for all  $r\ge0$ and $k\ge2$.
\end{enumerate}
\end{thm}
\begin{proof}
We give the proofs only for modular case, the cuspidal cases being similar. First, we see that $\fH^{\ord}_k$ is the subalgebra of $\End_\Lambda(m^\ord_{k,\Lambda})$ generated by the operators $T_q$, $U_p$ and $w_\ell$. But by Theorem \ref{thm:classical hida theory}, $m^\ord_{k,\Lambda}$ is independent of $k$. This shows that these algebras are independent of $k$.

The map $\fH^{\ord} \to \fH'^{\ord}$ in (1) is given as the composite
\[
\fH^{\ord} \to \Hom_\Lambda(m^\ord_\Lambda,\Lambda) \isoto  \fH'^{\ord}
\]
where the first map is induced by the pairing in (2), and the second map is the isomorphism of Theorem \ref{thm:classical hida theory} (2).  Let $X=\coker(\fH^{\ord} \to \fH'^{\ord})$, which is a finitely generated $\Lambda$-module.

Similarly, for any fixed $k$ and $r$, we have a map $\fH^{\ord}_{r,k} \to \fH'^{\ord}_{r,k}$ defined as the composite 
\[
\fH^{\ord}_{r,k} \to \Hom_{\Z_p}(m_{k}(\Gamma_r,\Z_p)^\ord,\Z_p) \isoto \fH'^{\ord}_{r,k},
\]
where the first map is given by $(-,-)_{r,k}$ and the second map is the isomorphism given by Lemma \ref{lem:classical duality}. This map is an isomorphism for every $r$ and $k$ if and only if (4) holds.

We have a commutative diagram to compare these maps:
\begin{equation}
\label{eq:hida square}
\xymatrix{
\fH^{\ord}/\omega_{r,k}\fH^{\ord} \ar[r] \ar@{->>}[d] & \fH'^{\ord}/\omega_{r,k}\fH'^{\ord} \ar[d]^-\wr\\
\fH^{\ord}_{r,k} \ar@{^(->}[r]& \fH'^{\ord}_{r,k}.
}
\end{equation}
The leftmost vertical arrow is surjective because the operators $T_q$, $U_p$ and $w_\ell$ map to the operators of the same name, the lower horizontal arrow is injective by Proposition \ref{prop:Atkin-Lehner duality}(1), and the rightmost vertical arrow is an isomorphism by Theorem \ref{thm:classical hida theory} (3). For $r=0$ and $k>2$, Proposition \ref{prop:Atkin-Lehner duality}(2) implies that the lower horizontal arrow is an isomorphism. This implies that the map
\[
\fH^{\ord}/\omega_{0,k}\fH^{\ord} \to \fH'^{\ord}/\omega_{0,k}\fH'^{\ord}
\]
is surjective for all $k> 2$. In other words, we have $X/\omega_{0,k}X=0$ for all $k>2$. The elements $\omega_{0,3},\omega_{0,4},\dots, \omega_{0,p+1}$ are in the $p-1$ different maximal ideals of $\Lambda$, so $X=0$ by Nakayama's lemma. This implies that $\fH^{\ord} \to \fH'^{\ord}$ is surjective. 

Returning to the diagram \eqref{eq:hida square} for arbitrary $r\ge0$ and $k\ge 2$, we see that the lower horizontal arrow is also surjective, proving (4). Since the map $\fH^{\ord} \to \fH'^{\ord}$ is the inverse limit (for $k$ fixed and $r$ increasing) of these maps, it is also an isomorphism proving (1). By the definition of the map $\fH^{\ord} \to \fH'^{\ord}$, this also proves (2). This shows that all the arrows except the leftmost vertical in \eqref{eq:hida square} are isomorphisms, so it is too, proving (3).
\end{proof}

\bibliographystyle{alpha}
\bibliography{june2018}

\def\cprime{$'$} \def\Dbar{\leavevmode\lower.6ex\hbox to 0pt{\hskip-.23ex
  \accent"16\hss}D} \def\cfac#1{\ifmmode\setbox7\hbox{$\accent"5E#1$}\else
  \setbox7\hbox{\accent"5E#1}\penalty 10000\relax\fi\raise 1\ht7
  \hbox{\lower1.15ex\hbox to 1\wd7{\hss\accent"13\hss}}\penalty 10000
  \hskip-1\wd7\penalty 10000\box7}
  \def\cftil#1{\ifmmode\setbox7\hbox{$\accent"5E#1$}\else
  \setbox7\hbox{\accent"5E#1}\penalty 10000\relax\fi\raise 1\ht7
  \hbox{\lower1.15ex\hbox to 1\wd7{\hss\accent"7E\hss}}\penalty 10000
  \hskip-1\wd7\penalty 10000\box7}
\begin{thebibliography}{WWE21}

\bibitem[AL70]{AL1970}
A.~O.~L. Atkin and J.~Lehner.
\newblock Hecke operators on {$\Gamma _{0}(m)$}.
\newblock {\em Math. Ann.}, 185:134--160, 1970.

\bibitem[BH93]{BH1993}
Winfried Bruns and J\"urgen Herzog.
\newblock {\em Cohen-{M}acaulay rings}, volume~39 of {\em Cambridge Studies in
  Advanced Mathematics}.
\newblock Cambridge University Press, Cambridge, 1993.

\bibitem[BP19]{BP2018}
Jo\"{e}l Bella\"{\i}che and Robert Pollack.
\newblock Congruences with {E}isenstein series and {$\mu$}-invariants.
\newblock {\em Compos. Math.}, 155(5):863--901, 2019.

\bibitem[Bro94]{brown1994}
Kenneth~S. Brown.
\newblock {\em Cohomology of groups}, volume~87 of {\em Graduate Texts in
  Mathematics}.
\newblock Springer-Verlag, New York, 1994.
\newblock Corrected reprint of the 1982 original.

\bibitem[Deo23]{Deo}
Shaunak~V. Deo.
\newblock {The Eisenstein ideal of weight $k$ and ranks of Hecke algebras}.
\newblock {\em Journal of the Institute of Mathematics of Jussieu}, page
  1–35, 2023.

\bibitem[FK24]{FK2024}
Takako Fukaya and Kazuya Kato.
\newblock On conjectures of {S}harifi.
\newblock {\em Kyoto J. Math.}, 64(2):277--420, 2024.

\bibitem[Gou92]{gouvea1992}
Fernando~Q. Gouv\^ea.
\newblock On the ordinary {H}ecke algebra.
\newblock {\em J. Number Theory}, 41(2):178--198, 1992.

\bibitem[Gre01]{greenberg2001}
Ralph Greenberg.
\newblock Iwasawa theory---past and present.
\newblock In {\em Class field theory---its centenary and prospect ({T}okyo,
  1998)}, volume~30 of {\em Adv. Stud. Pure Math.}, pages 335--385. Math. Soc.
  Japan, Tokyo, 2001.

\bibitem[Hid86a]{hida1986}
Haruzo Hida.
\newblock Galois representations into {${\rm GL}_2({\bf Z}_p[[X]])$} attached
  to ordinary cusp forms.
\newblock {\em Invent. Math.}, 85(3):545--613, 1986.

\bibitem[Hid86b]{hida1986a}
Haruzo Hida.
\newblock Iwasawa modules attached to congruences of cusp forms.
\newblock {\em Ann. Sci. \'Ecole Norm. Sup. (4)}, 19(2):231--273, 1986.

\bibitem[Hid93]{hida1993}
Haruzo Hida.
\newblock {\em Elementary theory of {$L$}-functions and {E}isenstein series},
  volume~26 of {\em London Mathematical Society Student Texts}.
\newblock Cambridge University Press, Cambridge, 1993.

\bibitem[Lec18]{lecouturier2016}
Emmanuel Lecouturier.
\newblock On the {G}alois structure of the class group of certain {K}ummer
  extensions.
\newblock {\em J. Lond. Math. Soc. (2)}, 98(1):35--58, 2018.

\bibitem[Lec21]{lecouturier2021}
Emmanuel Lecouturier.
\newblock Higher {E}isenstein elements, higher {E}ichler formulas and rank of
  {H}ecke algebras.
\newblock {\em Invent. Math.}, 223(2):485--595, 2021.

\bibitem[LW22]{LW}
Jaclyn Lang and Preston Wake.
\newblock A modular construction of unramified {$p$}-extensions of
  {$\Bbb{Q}(N^{1/p})$}.
\newblock {\em Proc. Amer. Math. Soc. Ser. B}, 9:415--431, 2022.

\bibitem[Maz72]{mazur1972}
Barry Mazur.
\newblock Rational points of abelian varieties with values in towers of number
  fields.
\newblock {\em Invent. Math.}, 18:183--266, 1972.

\bibitem[Maz77]{mazur1978}
B.~Mazur.
\newblock Modular curves and the {E}isenstein ideal.
\newblock {\em Inst. Hautes \'Etudes Sci. Publ. Math.}, (47):33--186 (1978),
  1977.

\bibitem[Mum08]{mumford2008}
David Mumford.
\newblock {\em Abelian varieties}, volume~5 of {\em Tata Institute of
  Fundamental Research Studies in Mathematics}.
\newblock Tata Institute of Fundamental Research, Bombay; by Hindustan Book
  Agency, New Delhi, 2008.
\newblock With appendices by C. P. Ramanujam and Yuri Manin, Corrected reprint
  of the second (1974) edition.

\bibitem[Oht93]{ohta1993}
Masami Ohta.
\newblock On cohomology groups attached to towers of algebraic curves.
\newblock {\em J. Math. Soc. Japan}, 45(1):131--183, 1993.

\bibitem[Oht99]{ohta1999}
Masami Ohta.
\newblock Ordinary {$p$}-adic \'etale cohomology groups attached to towers of
  elliptic modular curves.
\newblock {\em Compositio Math.}, 115(3):241--301, 1999.

\bibitem[Oht00]{ohta2000}
Masami Ohta.
\newblock Ordinary {$p$}-adic \'etale cohomology groups attached to towers of
  elliptic modular curves. {II}.
\newblock {\em Math. Ann.}, 318(3):557--583, 2000.

\bibitem[Oht03]{Ohta2003}
Masami Ohta.
\newblock Congruence modules related to {E}isenstein series.
\newblock {\em Ann. Sci. \'Ecole Norm. Sup. (4)}, 36(2):225--269, 2003.

\bibitem[Oht05]{ohta2005}
Masami Ohta.
\newblock Companion forms and the structure of {$p$}-adic {H}ecke algebras.
\newblock {\em J. Reine Angew. Math.}, 585:141--172, 2005.

\bibitem[Oht14]{ohta2014}
Masami Ohta.
\newblock Eisenstein ideals and the rational torsion subgroups of modular
  {J}acobian varieties {II}.
\newblock {\em Tokyo J. Math.}, 37(2):273--318, 2014.

\bibitem[Rib76]{ribet1976}
Kenneth~A. Ribet.
\newblock A modular construction of unramified {$p$}-extensions of {${\bf
  Q}(\mu _{p})$}.
\newblock {\em Invent. Math.}, 34(3):151--162, 1976.

\bibitem[{Sag}22]{sagemath}
The {Sage Developers}.
\newblock {\em {S}age{M}ath, the {S}age {M}athematics {S}oftware {S}ystem},
  2022.

\bibitem[Sha11]{sharifi2011}
Romyar Sharifi.
\newblock A reciprocity map and the two-variable {$p$}-adic {$L$}-function.
\newblock {\em Ann. of Math. (2)}, 173(1):251--300, 2011.

\bibitem[{Sta}]{stacks-project}
The {Stacks Project Authors}.
\newblock {\itshape Stacks Project}.
\newblock \url{http://stacks.math.columbia.edu}.

\bibitem[Wak23]{Wake-JEMS}
Preston Wake.
\newblock The {E}isenstein ideal for weight {$k$} and a {B}loch-{K}ato
  conjecture for tame families.
\newblock {\em J. Eur. Math. Soc. (JEMS)}, 25(7):2815--2861, 2023.

\bibitem[Wil88]{wiles1988}
A.~Wiles.
\newblock On ordinary {$\lambda$}-adic representations associated to modular
  forms.
\newblock {\em Invent. Math.}, 94(3):529--573, 1988.

\bibitem[WWE19]{PG4}
Preston Wake and Carl Wang-Erickson.
\newblock Deformation conditions for pseudorepresentations.
\newblock {\em Forum Math. Sigma}, 7:Paper No. e20, 44, 2019.

\bibitem[WWE20]{PG3}
Preston Wake and Carl Wang-Erickson.
\newblock The rank of {M}azur's {E}isenstein ideal.
\newblock {\em Duke Math. J.}, 169(1):31--115, 2020.

\bibitem[WWE21]{PG5}
Preston Wake and Carl Wang-Erickson.
\newblock The {E}isenstein ideal with squarefree level.
\newblock {\em Adv. Math.}, 380:Paper No. 107543, 62, 2021.

\end{thebibliography}

\end{document}